\documentclass[a4paper,10pt]{scrartcl}

\usepackage[fleqn]{amsmath}
\usepackage{amssymb}
\usepackage{amsthm}
\usepackage{stmaryrd}
\usepackage{bbm}
\usepackage[colorlinks]{hyperref}
\usepackage{mathtools}
\usepackage[numbers]{natbib}

\usepackage{enumitem,soul}
\setstcolor{purple}

\mathtoolsset{showonlyrefs}

\usepackage{pgfplots}

\DeclareMathOperator{\Var}{Var}

\newcommand{\bE}{\ensuremath{\mathbb{E}}}

\newcommand{\bN}{\ensuremath{\mathbb{N}}}

\newcommand{\bP}{\ensuremath{\mathbb{P}}}

\newcommand{\bR}{\ensuremath{\mathbb{R}}}

\newcommand{\bZ}{\ensuremath{\mathbb{Z}}}

\newcommand{\ind}{\ensuremath{\mathbbm{1}}}

\newcommand{\cE}{\ensuremath{\mathcal{E}}}

\newcommand{\cP}{\ensuremath{\mathcal{P}}}

\newcommand{\cR}{\ensuremath{\mathcal{R}}}

\newcommand{\norm}[1]{\left\Vert \, #1 \, \right\Vert}

\usepackage{mathtools}

\theoremstyle{plain}
\newtheorem{theorem}{Theorem}[section]  
\newtheorem{proposition}[theorem]{Proposition}

\newtheorem{lemma}[theorem]{Lemma}

\theoremstyle{definition}

\theoremstyle{remark}
\newtheorem{remark}{Remark}[section]

\numberwithin{equation}{section}

\title{Mixing for Poisson representable processes and consequences for the Ising model and the contact process}
\author{Stein Andreas Bethuelsen\footnote{Department of Mathematics, University of Bergen, Norway.}\quad Malin Palö Forsstr\"om\footnote{Chalmers University of Technology and University of Gothenburg, Sweden}}

\usepackage{scalerel,stackengine}
\newcommand\pig[1]{\scalerel*[5.5pt]{\Big#1}{%
 \ensurestackMath{\addstackgap[1.5pt]{\big#1}}}}
\newcommand\pigl[1]{\mathopen{\pig{#1}}}
\newcommand\pigr[1]{\mathclose{\pig{#1}}}

\begin{document}
\maketitle

\begin{abstract}
  \citet{JeffNinaMalin2024} recently introduced a large class of $\{0,1\}$-valued processes that they named Poisson representable. In addition to deriving several interesting properties for these processes, their main focus was determining which processes are contained in this class. 

  In this paper, we derive new characteristics for Poisson representable processes in terms of certain mixing properties. Using these, we argue that neither the upper invariant measure of the supercritical contact process on $\bZ^d$ nor the plus state of the Ising model on $\bZ^2$ within the phase transition regime is Poisson representable. Moreover, we show that on \( \mathbb{Z}^d,\) \( d \geq 2,\) any non-extremal translation invariant state of the Ising model cannot be Poisson representable. 
  Together, these results provide answers to questions raised in~\cite{JeffNinaMalin2024}.
\end{abstract}

\section{Introduction, main results and outline of the paper}\label{sec:main}

We first recall the definition of Poisson representable processes from~\cite{JeffNinaMalin2024}. Let $S$ be a finite or countably infinite set, and let $\nu$ be a $\sigma$-finite measure on $\cP(S) \setminus\{\emptyset \}$, where $\cP(S)$ is the power set of $S$. Consider the corresponding Poisson process with intensity measure $\nu$, denoted by $Y^{\nu}$. 
Thus, $Y^{\nu}$ is a random (possibly empty) collection $(B_j)_{j \in I}$ of non-empty subsets (perhaps with repetitions) of $S$. This generates a $\{0,1\}$-valued process $X^{\nu} =(X^{\nu}_i)_{i \in S}$ defined by letting 
\begin{equation}\label{eq:def} 
  X^{\nu}(i) \coloneqq 
  \begin{cases}
    1 & \text{ if } i \in \cup_{j\in I} B_j, \cr 
    0 & \text{ otherwise.}
  \end{cases} 
\end{equation}
Similarly to~\cite[Definition 1]{JeffNinaMalin2024}, we denote by $\cR(S)$ the collection of all processes $(X(i))_{i \in S}$ that are equal (in distribution) to $X^{\nu}$ for some intensity measure $\nu$. A process $X\in \cR(S)$ is said to be \emph{Poisson representable}.

As discussed thoroughly in~\cite[Section 1]{JeffNinaMalin2024}, many well-studied stochastic processes are Poisson representable, with the random interlacement being one notable example. Moreover, as concluded in~\cite[Theorem 3.1]{JeffNinaMalin2024}, all non-trivial stationary positively associated Markov chains on $\{0,1\}^{\bZ}$ are in $\cR(\bZ)$. (In fact, it contains a larger class of certain renewal processes, see~\cite[Theorem 3.5]{JeffNinaMalin2024}). Identifying $-1$ with $0$, it thus follows that the Ising model on $\bZ$ is Poisson representable for all parameter values. 
On the contrary, by~\cite[Theorem 6.1]{JeffNinaMalin2024}, tree-indexed Markov chains are not always in $\cR$, nor is the Ising model on $\bZ^d$, $d\geq 2$, see~\cite[Theorem 6.3]{JeffNinaMalin2024}. This latter result, however, was only proven to hold when the parameter $\beta$ is sufficiently small. Based on this, a natural open question, raised in~\cite[Question 2, Section 8]{JeffNinaMalin2024}, is if, for any $d\geq 2$ and $\beta>0$, the Ising model on $\bZ^d$ is in $\cR(\bZ^d)$? 

We write $\mu_{\beta}^{\pm}$ to denote the plus phase and minus phase of the Ising model with inverse temperature $\beta$, respectively. Further, we denote its critical parameter value by
\begin{equation}
  \beta_c \coloneqq \beta_c(S) \coloneqq \inf 
  \{ \beta>0 \colon \mu_{\beta}^{+} \neq \mu_{\beta}^{-} \}.
\end{equation}
It is well known that $\beta_c(\bZ^d) \in (0,\infty)$ whenever $d\geq 2$ and, in fact, that it equals $1/2 \log(1 + \sqrt{2})$ when $d=2$. As a reference to the precise definition and for the basic properties of the Ising model, see, e.g.,~\citet{friedli_velenik_2017}.

Following the convention that $-1$ is identified with $0$, our first main results partially answer whether the Ising model on $\bZ^d$ is Poisson representable.

\begin{theorem}\label{theorem:Ising2}
    Let $X \sim \mu_{\beta}^+$ on $\{0,1\}^{\bZ^2}$. For all \( \beta > \beta_c\) it holds that $X \notin \cR(\bZ^2)$.
\end{theorem}

\begin{theorem}\label{theorem: all middle Ising} 
  Consider $X \sim \alpha \mu_{\beta}^+ + (1-\alpha)\mu_{\beta}^-$ on $\{0,1\}^{\bZ^d}$ with $\alpha \in (0,1),$ $d \geq 2$. 
  Then, for any \( \beta > \beta_c,\) it holds that $X \notin \cR(\bZ^d)$.
\end{theorem}

The proof of Theorem~\ref{theorem:Ising2} is presented in Section~\ref{ssection:cmp}. It is based on a general characterization of Poisson representable processes concentrating on finite sets in terms of certain mixing properties; see Theorem~\ref{thm:PRmain}. 

As we discuss in Remark~\ref{rem:Ising2}, for $d=2$, Theorem~\ref{theorem: all middle Ising} follows by the same proof as that for Theorem~\ref{theorem:Ising2}. The extension to cover Theorem~\ref{theorem: all middle Ising} and general dimensions are treated in Section~\ref{section: middle ising}, where we present its proof. 
This is based on another general result, Theorem~\ref{theorem: variance and limit}, which states that the ergodic averages convergence in $L^2$ for any non-trivial translation invariant Poisson representable process.

Our last main result concerns the contact process and its so-called upper-invariant measure, denoted here by $\mu_{\lambda}$, which is a probability distribution on $\{0,1\}^S$. 
 See Section~\ref{sec:CP} for a precise definition of $\mu_{\lambda}$ and \cite{LiggettSIS1999} for a general reference to the contact process. Particularly, recall that the corresponding critical value is given by 
\begin{equation}
\lambda_c\coloneqq \lambda_c(S) \coloneqq \inf \{ \lambda >0 \colon \mu_{\lambda} \neq \delta_{\bar{0}} \},  
\end{equation} where $\delta_{\bar{0}}$ denotes the distribution concentrating on the "all zeros" configuration, $\bar{0} \in \{0,1\}^{S}$. It is well known that $\lambda_c(S) \in (0,\infty)$ whenever $S$ is countable infinite.

For any \( \lambda > \lambda_c,\) the measure $\mu_{\lambda}$ possesses the downward-FKG (or d-FKG) property, as concluded in~\citet{BergHaggstromKahn2006}. That is, writing $\bP$ for the distribution of $X$, for all $I \subset S$, the conditional distribution $\bP \bigl(\cdot \mid X(I) \equiv 0 \bigr)$ is positively associated with respect to events on $\{0,1\}^{S\setminus I}.$ In other words, for all increasing events $A$ and \( B, \) on $S\setminus I$, it holds that
\begin{equation}\label{eq:dfkg}
  \bP \bigl(A \cap B \mid X(I) \equiv 0  \bigr) \geq \bP \bigl( B \mid X(I) \equiv 0\bigr) \cdot \bP \bigl( A \mid X (I)\equiv 0 \bigr).
\end{equation}
As concluded in~\cite[Theorem 2.4]{JeffNinaMalin2024},  all Poisson representable processes have the d-FKG property. 
As in~\cite[Question 5, Section 8]{JeffNinaMalin2024} it is therefore natural to ask whether $X\sim \mu_{\lambda}$ is Poisson representable.
Again, we conclude that this is generally not the case for the contact process on $\bZ^d$, $d\geq1$. 
\begin{theorem}\label{theorem:cp1}
  Let $X\sim \mu_{\lambda}$ on $\{0,1\}^{\bZ^d}$ with \( d \geq 1\). 
  Then $X\in \cR(\bZ^d)$ if and only if $\lambda \leq \lambda_c$. 
\end{theorem}

Thus, the upper invariant measure of the contact process is Poisson representable only in the trivial case when it equals the distribution concentrating on $\bar{0}$.

\subsection*{Outline of the paper} 

In the next section, we first state and prove Theorem~\ref{thm:PRmain}. In the following subsection, we show how to apply this to prove Theorem~\ref{theorem:Ising2} and Theorem~\ref{theorem:cp1}. Moreover, in several remarks, we discuss possible extensions of this approach that, among others, shed light on additional questions raised in~\cite{JeffNinaMalin2024}. 
Section~\ref{section: middle ising} is devoted to the proofs of~Theorem~\ref{theorem: all middle Ising} and~Theorem~\ref{theorem: variance and limit}. 
In the last section, Section~\ref{sec:CP}, we present the proof of Theorem~\ref{theorem:cp2}, which provides a mixing result of independent interest for the supercritical contact process that we apply in the proof of Theorem~\ref{theorem:cp1}.

\section{Proofs of Theorem~\ref{theorem:Ising2} and Theorem~\ref{theorem:cp1} }\label{section: new versions}

This section gives the detailed proofs of Theorem~\ref{theorem:Ising2} and Theorem~\ref{theorem:cp1}. We first present a characterization of Poisson representable processes with an intensity measure that concentrates on finite sets, see Theorem~\ref{thm:PRmain} below, on which these proofs hinge.

\subsection{Mixing for Poisson representable processes}\label{section: properties}

Recall that if $X \in \cR(S)$, then it can be constructed as detailed in~\eqref{eq:def}. 
On the same probability space, for any $\Gamma \subset \cP(S)$, we can also construct the process $X^{(\Gamma)}$ given by
\begin{equation}\label{eq:coupling_construction}
X^{(\Gamma)}(i) = 
\begin{cases}
  1 & \text{ if } i \in \cup_{j\in I_{\Gamma}} B_j, \cr 
  0 & \text{ otherwise,}
\end{cases} 
\end{equation}
where $I_{\Gamma} = \{ j \in I \colon B_j \in \Gamma \} \subseteq I (= I_{\cP(S)})$. 

Note that, for any $\Gamma \subset \cP(S)$, we have that $X^{(\Gamma)} \in \cR(S)$ with intensity measure $\nu|_{\Gamma}$. Particularly, $X=X^{(\cP(S))}$ and $\nu=\nu |_{\cP(S)}$. Moreover, note that in principle, it may be that $\nu |_{\Gamma}$ equals the trivial measure that assigns no weight to any subset of $\cP(S)$. In this latter case, $X^{(\Gamma)} \sim \delta_{\bar{0}}$.

\begin{lemma}\label{lemma: transfer to limit new}
    Consider $X \in \cR(S)$ and let $(\Gamma_n)$ be either an increasing or decreasing sequence of subsets of $\cP(S)$, with $\Gamma= \bigcup \Gamma_n$ (if increasing) or $\Gamma=\bigcap \Gamma_n $ (if decreasing).
    Then the weak limit of $X^{(\Gamma_n)}$ as $n\rightarrow \infty$ equals $X^{( \Gamma)}$.\ 
\end{lemma}

\begin{proof}[Proof of Lemma~\ref{lemma: transfer to limit new}] 
    This follows by the above construction and basic set theory.  
\end{proof}

Now let
\begin{equation*}
    \cP(S)^{(<\infty)} \coloneqq \bigl\{ \Delta \in \cP(S) \colon |\Delta|<\infty \bigr\}
\end{equation*}
and, for $X \in \cR(S)$ with corresponding intensity measure \( \nu,\) let 
\begin{equation*}
    \nu^{(<\infty)} \coloneqq \nu|_{\cP(S)^{(<\infty)}} \quad \text{and} \quad X^{(<\infty)} \coloneqq X^{\nu^{(<\infty)}}.
\end{equation*} 
Thus, the process $X^{(<\infty)} \in \cR(S)$ has an intensity measure that concentrates on finite sets. 
More generally, for $Q \subset S$, we consider
\begin{align}
  &\cP(S)^{(<\infty,Q)} \coloneqq \bigl\{ \Delta \in \cP(S) \colon |\Delta \cap Q|<\infty \bigr\},
  \\ &\nu^{(<\infty,Q)} \coloneqq \nu|_{\cP(S)^{(<\infty,Q)} }
  \quad \text{and} \quad 
  X^{(<\infty,Q)} \coloneqq X^{\nu^{(<\infty,Q)}}.
\end{align} 
Note that, with $Q=S$, we obtain that
\begin{equation*}
    \cP(S)^{(<\infty,S)}=\cP(S)^{(<\infty)} \quad \text{and} \quad X^{(<\infty,S)} = X^{(<\infty)}. 
\end{equation*}
The following lemma gives a first characterization of the $X^{(<\infty)}$-process.

\begin{lemma}\label{lem:PRmain}
  Let $X \in \cR(S)$ and $(Q_j)_{j \in J}$ be a partition of $S$ with $|J|<\infty$. Then $X= X^{(<\infty)}$ if and only if $X = X^{(<\infty,Q_j)}$ for each $j \in J$.
\end{lemma}

Our main result of this subsection is the following characterization of Poisson representable processes concentrating on finite sets.

\begin{theorem}\label{thm:PRmain}
  Let $X \in \cR(S),$ and let $(Q_j)_{j \in J}$ be a partition of $S$ with $|J|<\infty$. 
  Then ${X = X^{(<\infty)}}$ if and only if, for any increasing sequence $(S_n)_{n\geq 1}$ of sets such that $ \bigcup_{n\geq 1} S_n = S$, it holds that, with respect to weak convergence, 
  \begin{equation}\label{thm:PRmainEQ}
    \lim_{n \rightarrow \infty} \lim_{m\rightarrow \infty} \bP \bigl( X \in \cdot \mid X(Q_i \cap S_m\cap S_n^c) \equiv 0 \bigr) 
  = \bP (X \in \cdot ).
  \end{equation}

\end{theorem}

\begin{proof}
  This is a direct consequence of Lemma~\ref{lemma: transfer to limit new} and Lemma~\ref{lem:PRmain}. Indeed, by the so-called restriction theorem for Poisson processes ~\cite[Theorem 5.3]{LastPenrose}, for any \( j \in J, \) it holds that  
  $X^{\nu|_{\cP( S\setminus(Q_j \cap S_m\cap S_n^c))}}$ equals $\bP \bigl( X \in \cdot \mid X(Q_j\cap S_m \cap S_n^c) \equiv 0 \bigr)$ in distribution. 
  Now note that for any fixed \( j \in J \) and \( n \geq 1, \) the sequence $\bigl(\mathcal{P}(S\setminus(S_m \cap Q_j \cap S_n^c)) \bigr)_{m \geq 1}$ is decreasing and converges to $\mathcal{P}(S\setminus( Q_j \cap S_n^c)).$
  Moreover, for fixed \( j \in J, \) the sequence $\bigl( \mathcal{P}(S\setminus( Q_j \cap S_n^c) )\bigr)_{n \geq 1}$ is increasing and converges to $\cP(S)^{(<\infty,Q_j)}$. Hence, by Lemma~\ref{lemma: transfer to limit new}, the identity \eqref{thm:PRmainEQ} is equivalent to that $X=X^{(<\infty,Q_j)}$ for each $j \in J$. The conclusion thus follows by Lemma~\ref{lem:PRmain}.
\end{proof}

\begin{proof}[Proof of Lemma~\ref{lem:PRmain}]
It is immediate from the construction that, if $X= X^{(<\infty)}$, then also $X = X^{(<\infty,Q_j)}$ for all $j \in J$.
For the other direction, note that if $X = X^{(<\infty,Q_j)}$ for each $j \in J$, then a.s., by the construction as in~\eqref{eq:coupling_construction}, there are no element $B_i \subset S$ such that, for some $j \in J$, $|B_i \cap Q_j| = \infty$. Hence, a.s.\ $|B_i| <\infty$ for all $i \in I$, implying in particular that  $X= X^{(<\infty)}$.
\end{proof}

\begin{remark}
    For $X=X^{\nu} \in \cR(S)$ we note that $\nu= \nu^{(<\infty)} + \nu^{(=\infty)}$ where $ \nu^{(=\infty)} \coloneqq \nu|_{\cP(S) \setminus \cP(S)^{(<\infty)}}$. 
    As concluded in~\cite[Theorem 7.3]{JeffNinaMalin2024} for stationary processes on $\bZ$, the process $X$ is ergodic if and only if $\nu^{(=\infty)}$ concentrates on sets with zero density. Moreover, if $\nu^{(=\infty)}$ assigns no weight to any subset of $\cP(S)$, then~\cite[Theorem 7.5]{JeffNinaMalin2024} implies that the process $X^{\nu}=X^{\nu^{(<\infty)}}$ is a Bernoulli Shift, i.e., a factor of i.i.d.’s. As argued in~\cite[Theorem 7.7]{JeffNinaMalin2024} this latter result extends to processes on $\bZ^d$ with $d\geq 2$.
\end{remark}

\subsection{(Lack of) mixing for the contact process and the Ising model}\label{ssection:cmp}

Our intuition behind why neither the plus phase of the Ising model nor the upper invariant measure of the contact process are Poisson representable in their phase transition regimes is that their conditional distributions are, in some sense, "too correlated" for having a Poissonian construction satisfying the so-called restriction theorem~\cite[Theorem 5.3]{LastPenrose}. 
To make this intuition into a rigorous proof, we apply the characterization given by Theorem~\ref{thm:PRmain} for Poisson representable processes to construct a contradiction.

To conclude Theorem~\ref{theorem:Ising2}, we first recall two properties for the so-called Schonmann projection, \( Z=(Z(i))_{i \in \bZ} \), of the Ising model on $\bZ^2$ obtained from $X \sim \mu_{\beta}^+$ by letting $Z(i)=X(i,0)$.

\begin{lemma}\label{lem:Schonmann}
  For any $\beta>\beta_c$, with respect to weak convergence, the following holds.
  \begin{enumerate}[label=(\alph*)]
    \item \label{item: two sided nomixing} \( \displaystyle \lim_{n \rightarrow \infty} \lim_{m \rightarrow \infty} \mu_{\beta}^{+} \pigl( Z \in \cdot \mid Z\bigl([-m,m] \setminus [-n,n]\bigr) \equiv -1  \pigr) = \mu_{\beta}^{-}(Z \in \cdot )\)
    \item \label{item: first wetting}\( \displaystyle \lim_{n \rightarrow \infty} \lim_{m \rightarrow \infty} \mu_{\beta}^{+} (Z \in \cdot \mid Z\bigl( (-m,-n) \bigr) \equiv -1 ) = \mu_{\beta}^{+} (Z \in \cdot )\) 
  \end{enumerate}  
\end{lemma}

That Lemma~\ref{lem:Schonmann}\ref{item: two sided nomixing} holds follow by~\citet[Lemma 1]{SchonmannNGibbs1989}, whereas~Lemma~\ref{lem:Schonmann}\ref{item: first wetting} is a direct consequence of~\citet[Theorem 3.3]{BethuelsenConache2018}. From these properties and Theorem~\ref{thm:PRmain}, we conclude that $Z$ cannot be Poisson representable. 
 
\begin{proposition}
\label{prop:schonmann}
    Let $\beta>\beta_c$. Then, identifying $-1$ with $0$,  
  it holds that $Z \notin \cR(\bZ)$.  
\end{proposition}

\begin{proof}
  Assume for contradiction that $Z \in \cR(\bZ)$ and, for \( n \geq 1, \) let $S_n \coloneqq [-n,n].$ 
  Then, considering the the trivial partition where $(Q_j)_{j\in J} = (Q_1)$ with $Q_1=\bZ$, by Lemma~\ref{lem:Schonmann}\ref{item: two sided nomixing} and Theorem~\ref{thm:PRmain}, we have that $Z \neq Z^{(<\infty)}$ since \( \mu_{\beta}^+ \neq \mu_{\beta}^-\) when $\beta>\beta_c$.  
  
  On the other hand, consider the partition $(Q_1,Q_2) \coloneqq (\bN, \bZ\smallsetminus \bN)$ of $\bZ.$ From Lemma~\ref{lem:Schonmann}\ref{item: first wetting} and the symmetry of the model, it follows that~\eqref{thm:PRmainEQ} holds with respect to both $Q_1$ and $Q_2$.
  Thus, Theorem~\ref{thm:PRmain} says that $Z = Z^{(<\infty)}$, leading to a contradiction. Consequently, it cannot be that $Z \in \cR(\bZ)$.  
\end{proof}

\begin{proof}[Proof of Theorem~\ref{theorem:Ising2}]
  Let $X\sim \mu_{\beta}^+$ on $\bZ^2$ and assume for contradiction that $X \in \cR(\bZ^2)$. Then~\cite[Lemma 2.14(a)]{JeffNinaMalin2024} says that, for any $\Lambda \subset \bZ^2$, we have \( X|_{\Lambda} \in \cR(\Lambda).\) However, this stands in contradiction to Proposition~\ref{prop:schonmann}. Thus, $X \notin \cR(\bZ^2)$. 
\end{proof}

The following remarks detail further consequences and possible extensions of the above arguments for the Ising model.

\begin{remark}\label{rem:Ising1} 
The plus-phase of the Ising model equals the limit (with respect to weak convergence) of $\mu_{n,\beta}^+$ as $n\rightarrow \infty$, where \( \mu_{n,\beta}^+\) is the Ising model with interaction parameter \( \beta \) on \( \Lambda_n \coloneqq [-n,n]^d\cap \mathbb{Z}^d\) with plus boundary conditions. Since, by Theorem~\ref{theorem:Ising2}, we have that $X \notin \cR(\bZ^2)$, it therefore follows by~\cite[Lemma 2.22]{JeffNinaMalin2024} for $d=2$ that $\mu_{n,\beta}^+ \notin \cR([-n,n]^2)$ for all sufficiently large \( n. \)
 \end{remark}

 \begin{remark}\label{rem:Ising2}
  Let $\alpha \in (0,1)$ and consider the process $X \sim \alpha \mu_{\beta}^+ + (1-\alpha)\mu_{\beta}^-$ on $\{0,1\}^{\bZ^2}$. Then, the statement of Lemma~\ref{lem:Schonmann} still holds for the corresponding projection onto $\bZ \times \{0\}$. Therefore, the proof of Proposition~\ref{prop:schonmann} and thus also that of Theorem~\ref{theorem:Ising2}, extends to this case. As a consequence, $X \notin \cR(\bZ^2)$. An alternative proof of this statement, which also extends to higher dimensional lattices, is given in the next section. 
\end{remark}

\begin{remark}\label{rem:Ising3}
  Generally, if $X \sim \mu_{\beta}^+$ is Poisson representable, then $X^{(<\infty)}\sim \mu_{\beta}^-$ is Poisson representable too, as follows by Lemma~\ref{lemma: transfer to limit new}. Particularly, in the uniqueness phase $\beta<\beta_c$ (and also at $\beta_c$ for the model on $\bZ^d$), if $X$ is Poisson representable, then, by Lemma~\ref{lemma: transfer to limit new}, its intensity measure necessarily concentrates on finite sets.
  On the other hand, for any value of $\beta>0$, if the minus phase of the Ising model is not Poisson representable, then neither is the plus phase nor any other phase. Unfortunately, we do not see how the arguments of this section can be used to determine whether the minus phase is Poisson representable or not.
\end{remark}

\begin{remark}\label{rem:Ising4} 
  The arguments of this subsection may be extended to other graphs, as we outline next, focussing on the case where $S=\bZ^d$, $d\geq 2$. 
  
  The contrasting mixing behavior seen in Proposition~\ref{prop:schonmann} are well known for models from statistical mechanics.  
  In particular, Lemma~\ref{lem:Schonmann}\ref{item: two sided nomixing} was derived in~\cite{SchonmannNGibbs1989} to conclude that the Schonmann projection $(Z(i))_{i \in \bZ}$ is non-Gibbsian. The latter conclusion was later extended to the projection of the $d$-dimensional Ising model onto a $d-1$ layer in~\cite{FernandezPfisterNGibbs1997}. Presumably, such projections satisfy the natural extension of~Lemma~\ref{lem:Schonmann}\ref{item: two sided nomixing} too. %
  On the other hand, in~\cite{MaesRedigMoffaertNGibbs1999} a general approach was laid out for proving that projections of Gibbs measure onto a sufficiently decimated $(d-1)$-dimensional layer preserves the Gibbsian property. Their approach also implies that these models have mixing properties reminiscent of those of Lemma~\ref{lem:Schonmann}\ref{item: first wetting}.  
  More concretely, it was concluded in~\cite[Theorem 4.2]{MaesRedigMoffaertNGibbs1999} that, for $X\sim \mu^+_{\beta}$ on $\bZ^2$ and $\beta$ sufficiently large, the corresponding Schonmann projection $\bigl(Z(i)\bigr)_{i \in \bZ}$ satisfy~\eqref{thm:PRmainEQ}  with respect to the partitioning $Q\coloneqq(Q_0,Q_1,Q_2,Q_3)$ given by  
  \begin{equation}
    Q_i \coloneqq \{ x \in \bZ \colon x \equiv i \mod 4 \},\quad i=0,1,2,3,
  \end{equation}
  (see also~\cite[Theorem 1]{LorincziVelde1994} for similar results). 
  Presumably, this approach can also be extended to higher dimensional lattices in the supercritical regime $\beta>\beta_c$. (Theorem~\ref{theorem:cp2} below provides such an extension for the supercritical contact process). In that case, combined with the observations of the previous paragraph and the arguments of this section, this would imply that $X\sim \mu_{\beta}^+$ on $\{0,1\}^{\bZ^d}$ is not Poisson representable. 
\end{remark}

For the contact process, we prove in Section~\ref{sec:CP} that $\mu_{\lambda}$ satisfies contrasting mixing behavior similar to those seen in Proposition~\ref{prop:schonmann} and described precisely in the following statement.

\begin{theorem}\label{theorem:cp2}
    Let $X\sim \mu_{\lambda}$ on $\{0,1\}^{\bZ^d}$, where \( d \geq 1\) and $\lambda > \lambda_c$.  
  \begin{enumerate}[label=(\alph*)]
    \item \label{item: cp2a}For any $x \in \bZ^d$,
    \begin{equation}\label{eq:lem:cp111111111111}
      \lim_{n \rightarrow \infty} \lim_{m\rightarrow \infty}  \mu_{\lambda}\pigl( X(x) = 0 \mid X\bigl( [-m,m]^{d-1} \times (-m, -n) \bigr) \equiv 0 \pigr)  
       = \mu_{\lambda}\bigl(X(x)=0 \bigr).
    \end{equation} 
    \item \label{item: cp2b}For any $n \in \bN$, 
    \begin{equation}\label{eq:lem:cp1.2}
     \lim_{m\rightarrow \infty} \mu_{\lambda}\pigl(X\bigl( [-n,n]^d \bigr) \equiv 0  \mid X\bigl([-m,m]^{d} \setminus [-n,n]^d \bigr) \equiv 0 \pigr) = 1.
    \end{equation} 
  \end{enumerate}
\end{theorem}

\begin{remark}
   For processes satisfying the d-FKG property, limits as those in Theorem~\ref{theorem:cp2} are well-defined. To see this, recall \eqref{eq:dfkg} and note that, for $\Delta \subset S$ finite and any decreasing event $A$ on $\{0,1\}^{\Delta}$,
   \begin{equation}
       \bP \bigl(A \mid X(I_n) \equiv 0  \bigr)
   \end{equation}
   is increasing in $n$ for any increasing sequence $(I_n)$ contained in $S\setminus \Delta$.
\end{remark}

Now, armed with Theorem~\ref{thm:PRmain} and Theorem~\ref{theorem:cp2}, we move on to the proof of the main result for the contact process; Theorem~\ref{theorem:cp1}. 

\begin{proof}[Proof of Theorem~\ref{theorem:cp1}] 
  Let $X \sim \mu_{\lambda}$ where $\mu_{\lambda}$ is the upper invariant measure of the contact process on $\bZ^d$, $d\geq 1$ and $\lambda>0$.

  If $\lambda < \lambda_c$, then by definition $\mu_{\lambda}= \delta_{\bar{0}}$ and thus assigns all weight to the all zeros configuration. It is also well-known that 
  $\mu_{\lambda_c}= \delta_{\bar{0}}$, see e.g.~\ \cite[Theorem 2.25]{LiggettSIS1999}. Hence, for $\lambda \leq \lambda_c$, it follows that \( X \in \mathcal{R}(\bZ^d)\) with corresponding measure \( \nu \equiv 0.\) 
  
  Now, consider the more interesting case that $\lambda > \lambda_c$. 
  Then Theorem~\ref{theorem:cp2} stands in contrast to Theorem~\ref{thm:PRmain}, from which we conclude that $X \notin \cR(\bZ^d)$.  Indeed, assume that $X \in \cR(\bZ^d)$ and let $S_n=[-n,n]^d$, $n\geq 1$. 
  Then, by considering the trivial partitioning where $Q=\{ \bZ^d\}$ in Theorem~\ref{thm:PRmain}, it follows by Theorem~\ref{theorem:cp2}\ref{item: cp2b}, that $X \neq X^{(<\infty)}$ since $X$ violates \eqref{thm:PRmainEQ}. On the other hand, if we consider the partitioning $(Q_j)_{j=1}^{2d}$ of $\bZ^d$ into its quadrants, then it follows by Theorem~\ref{theorem:cp2}\ref{item: cp2a} and symmetry of the model that \eqref{thm:PRmainEQ} holds with respect to each $Q_i$, $i=1,\dots,2d$. In particular, Theorem~\ref{thm:PRmain} yields that $X = X^{(<\infty)}$, leading to the aforementioned contradiction. Consequently, it cannot be that $X \in \cR(\bZ^d)$. 
\end{proof} 

\begin{remark} 
  As previously noted, the d-FKG property is a unifying property for Poisson representable processes. This property was first introduced in~\citet{BergHaggstromKahn2006} to study certain percolation models. Notably, for any countable-infinite graph $(S,E)$ and any $p\in [0,1]$, the d-FKG property was therein shown to hold for the process $X \coloneqq \bigl(X(i)\bigr)_{i \in S}$ obtained by setting $X(i)=1$ if and only if $i$ is contained in an infinite component of the corresponding ordinary percolation process. 
  We are confident that similar reasoning as for the proof of Theorem~\ref{theorem:Ising2} can be applied to this model when $S=\bZ^2$, from which one would conclude that it is not Poisson representable when $p>p_c$. For this, presumably, the equivalence of the "one-sided mixing" of~Lemma~\ref{lem:Schonmann}\ref{item: first wetting} can be shown to hold using the ideas of~\cite[Theorem 3.3]{BethuelsenConache2018} and the large deviation bounds of~\citet[Theorem 5]{DurrettSchonmann1988}. Similarly, we also expect that the process \( X \) is not "two-sided mixing" in the sense that it satisfies the equivalent of Theorem~\ref{theorem:cp1}\ref{item: cp2b}. Moreover, we have no reason to believe that, for this model, this would be any different in higher dimensions. This would answer~\cite[Question 4, Section 8]{JeffNinaMalin2024} in the negative. 
  We also anticipate that analogous statements can be proven for the more general FK-percolation model (or even the Fuzzy Potts model), which are also covered in the work of~\cite{BergHaggstromKahn2006} and shown to satisfy the d-FKG property. 
\end{remark}

\section{Proof of Theorem~\ref{theorem: all middle Ising}}\label{section: middle ising}

In this section, we give the detailed proof of Theorem~\ref{theorem: all middle Ising}. As mentioned earlier, this relies on the general property that non-trivial translation invariant Poisson representable processes on $\bZ^d$ cannot be bimodal; see Theorem~\ref{theorem: variance and limit} below. 

\subsection{Impossibility of bimodality for Poisson representable processes}

In~\cite[Theorem 5.2]{JeffNinaMalin2024}, it was shown that the Curie-Weiss model, i.e., the Ising model on a complete graph on \( n \) vertices, is not in \( \mathcal{R}\) for any \( \beta> \beta_c \) and \( n\) sufficiently large. The main idea of the proof of this result was to show that if \( X = (X_1,\dots, X_n ) \) is permutation invariant and \( X \in \mathcal{R}, \) then \( \bar X \) cannot be bimodal in a certain sense (see~\cite[Theorem 5.1]{JeffNinaMalin2024}), where \( \bar X \coloneqq (X_1+X_2+\dots+ X_n)/n.\) It is natural to ask if the assumption on permutation invariance can be loosened in this result to instead e.g.\ assume only translation invariance. 
The following theorem concludes this in the affirmative for processes \( X \in \cR(\bZ^d)\) with $d\geq 1$. 
For this, for any $n\ge 1,$ we denote by \( \bar X_n \coloneqq |\Lambda_n|^{-1}\sum_{i \in \Lambda_n} X(i) \), where we recall that \( \Lambda_n \coloneqq [-n,n]^d\cap \mathbb{Z}^d\).

\begin{theorem}\label{theorem: variance and limit}
  Let $d\geq 1$, and let \( X \in \cR(\bZ^d)\) be translation invariant.  
   If, for some \( c \in (0,1)\),
  \begin{equation}\label{eq: limit assumption}
    \lim_{n\to \infty} \bP(\bar X_n \geq c) = 0,
  \end{equation}  
  then  
   $  \lim_{n \to \infty} \Var(\bar X_n) = 0$.
\end{theorem}

We postpone the proof of Theorem~\ref{theorem: variance and limit} to the next section. 
It uses translation invariance as a symmetry, and the idea can presumably, with some work, be extended to more general lattices with other symmetries.

\begin{proof}[Proof of Theorem~\ref{theorem: all middle Ising}]
  The conclusion of the theorem follows immediately from Theorem~\ref{theorem: variance and limit} by noting that for $\beta>\beta_c$ the measures $\mu_{\beta}^+$ and $\mu_{\beta}^-$ satisfies the assumptions of Theorem~\ref{theorem: variance and limit} and $\bE_{\mu_{\beta}^+}(X(0)) \neq \bE_{\mu_{\beta}^-}(X(0)$. 
\end{proof}

\subsection{Proof of Theorem~\ref{theorem: variance and limit}}\label{section: proof of variance}

In the remainder of this section, we give a proof of Theorem~\ref{theorem: variance and limit}. The proof of this theorem builds on the following technical lemma. 
To state it precisely, as in \cite{JeffNinaMalin2024}, we let 
\begin{equation*}
  \mathcal{S}_i 
 \coloneqq \bigl\{ \Delta \in \mathcal{P}(S)\smallsetminus \{ \emptyset \} \colon i \in \Delta \bigr\}, \quad i \in S. 
\end{equation*}
Further, for $\Delta\subseteq S$, we let
\begin{equation*}
  \mathcal{S}_\Delta^\cup \coloneqq \mathcal{S}_\Delta^{\cup,S} \coloneqq \bigcup_{i\in \Delta} \mathcal{S}_i 
  \quad \text{and} \quad \mathcal{S}_\Delta^\cap \coloneqq \bigcap_{i\in \Delta} \mathcal{S}_i .
\end{equation*}
This notation will be useful to us because it connects probabilities involving \( X^\nu \) with the measure \( \nu \) in the sense that for any set \( \Delta \subseteq S, \) one has
\begin{equation}\label{eq: nu P connection}
  P\bigl( X^\nu(\Delta) \equiv 0 \bigr) = e^{-\nu(\mathcal{S}_\Delta^\cup)}.
\end{equation}

\begin{lemma}\label{lemma: mean is one}
  In the setting of Theorem~\ref{theorem: variance and limit}, we have
  \begin{equation}\label{eq: goal 2}
    \lim_{n \to \infty} |\Lambda_n|^{-2}\sum_{i,j \in \Lambda_n } e^{\nu(\mathcal{S}_{\{ i,j\}}^{\cap})} =1.
  \end{equation}
\end{lemma}

\begin{proof}[Proof of Theorem~\ref{theorem: variance and limit}]
  Note first that
  \begin{align*}
    \mathbb{E}[\bar X_n ] = \mathbb{E}\bigl[X(0)\bigr] = 1-e^{-\nu(\mathcal{S}_0)} \leq 1.
  \end{align*} 
  Further, note that
  \begin{align*}
    &\left|\Lambda_n \right|^{2} \mathbb{E}[\bar X_n^2 ] = \mathbb{E}\Bigl[\bigl(\sum_{i \in \Lambda_n} X(i)\bigr)^2\Bigr] 
    = 
    \left|\Lambda_n \right| \mathbb{E}[X(0)^2] + \sum_{i,j \in \Lambda_n \colon i \neq j} \mathbb{E}\bigl[X(i) X(j)\bigr].
  \end{align*}
  Here
  \begin{align*}
    \mathbb{E}[X(0)^2]=\mathbb{E}[X(0)] \leq 1,
  \end{align*}
  and
  \begin{align*}
    &\sum_{i,j \in \Lambda_n \colon i \neq j} \mathbb{E}\bigl[X(i) X(j)\bigr] 
    =
    \sum_{i,j \in \Lambda_n \colon i \neq j} \bigl( 1 -2e^{-\nu(\mathcal{S}_1) }+e^{-2\nu(\mathcal{S}_1)+\nu(\mathcal{S}_{\{ i,j\}}^{\cap})} \bigr)
    \\&\qquad= 
    \left|\Lambda_n \right| \bigl(\left|\Lambda_n \right|-1\bigr) \bigl( 1 -2e^{-\nu(\mathcal{S}_0) } \bigr) + e^{-2\nu(\mathcal{S}_0)} \sum_{i,j \in \Lambda_n \colon i \neq j} e^{ \nu(\mathcal{S}_{\{ i,j\}}^{\cap})} .
  \end{align*}
  Combining the previous equation, we get 
    \begin{align*}
         &\Var(\bar X_n^2)
         =  \left|\Lambda_n \right|^{-1}  e^{-\nu(\mathcal{S}_0)}
         + e^{-2\nu(\mathcal{S}_0)} \biggl( |\Lambda_n|^{-2}\sum_{i,j \in \Lambda_n \colon i \neq j} e^{ \nu(\mathcal{S}_{\{ i,j\}}^{\cap})} 
         -
         1  \biggr).
    \end{align*} 
  Using~Lemma~\ref{lemma: mean is one}, the desired conclusion immediately follows.
\end{proof}

We now continue with the proof of Lemma~\ref{lemma: mean is one}, which relies on the following statement. 
For this, with \( \delta \in (0,1) \) and \( n \geq 1,\) let 
  \begin{equation}\label{eq:S_n_delta}
  \mathcal{S}^{n,\delta} \coloneqq \bigr\{ \Delta \subseteq S \colon |\Delta \cap \Lambda_n| \geq \delta |\Lambda_n|\bigr\}.
  \end{equation}

\begin{lemma}\label{lemma: variance and limit}
  In the setting of Theorem~\ref{theorem: variance and limit}, it holds that \( {\lim_{n \to \infty}\nu(\mathcal{S}^{n,\delta}) =0.}\)  
\end{lemma}

\begin{proof}[Proof of Lemma~\ref{lemma: mean is one}]
    Let \( \delta \in (0,1), \) and \( n \geq 1\). Then any set \( \Delta \subseteq \mathbb{Z}^d \) such that \( \Delta \notin \mathcal{S}^{n,\delta}\) satisfies \( |\Delta \cap \Lambda_n| \leq \delta |\Lambda_n|.\) From this it follows that
    \begin{equation}\label{eq: step 1 to goal 2}
        \nu(\mathcal{S}_i) \geq \nu( \mathcal{S}_i \smallsetminus \mathcal{S}^{n,\delta})
        \geq 
        \sum_{j \in \Lambda_n\smallsetminus \{ i \}} \nu( \mathcal{S}_{\{ i,j \}}^\cap\smallsetminus \mathcal{S}^{n,\delta})/(\delta |\Lambda_n|) .
      \end{equation} 
  For \( r > 0 \) and \( i \in \Lambda_n,\) let 
  \[
    K_{n,r}(i) \coloneqq \bigl\{ j \in \Lambda_n \smallsetminus \{ i \} \colon \nu(\mathcal{S}_{\{i,j\}}^\cap)>r \bigr\}.
  \]  
  Using~\eqref{eq: step 1 to goal 2}, it then follows that
  \begin{align*}
    &\nu(\mathcal{S}_i) 
    \geq 
    \sum_{j \in K_{n,r}(i)} \nu( \mathcal{S}_{\{ i,j \}}^\cap\smallsetminus \mathcal{S}^{n,\delta})/(\delta |\Lambda_n|)
    \\&\qquad= 
    \sum_{j \in K_{n,r}(i)} \nu( \mathcal{S}_{\{ i,j \}}^\cap )/(\delta |\Lambda_n|)
    -\sum_{j \in K_{n,r}(i)} \nu( \mathcal{S}_{\{ i,j \}}^\cap\cap \mathcal{S}^{n,\delta})/(\delta |\Lambda_n|)
    \\&\qquad> 
    |K_{n,r}(i)| r/(\delta |\Lambda_n|)
    -|K_{n,r}(i)|\nu( \mathcal{S}^{n,\delta})/(\delta |\Lambda_n|)
    =
    \frac{|K_{n,r}(i)|}{\delta |\Lambda_n|} \bigl( r 
    - \nu( \mathcal{S}^{n,\delta}) \bigr) 
  \end{align*} 
  and hence
  \begin{equation*}
    |K_{n,r}(i)|/|\Lambda_n| < \frac{\delta \nu(\mathcal{S}_i)}{r-\nu(\mathcal{S}^{n,\delta})}
  \end{equation*}
  whenever \( r > \nu(\mathcal{S}^{n,\delta}).\) 
  Since \( K_{n,r}(i)\) does not depend on \( \delta,\) it follows from Lemma~\ref{lemma: variance and limit} that for any \( r>0,\) we have
  \begin{align*}
    \lim_{n \to \infty}|K_{n,r}(i)|/|\Lambda_n| = 0.
  \end{align*} 
  Now note that
  \begin{align*}
    &\sum_{i,j \in \Lambda_n \colon i \neq j} e^{\nu(\mathcal{S}_{\{ i,j\}}^{\cap})} 
    =
    \sum_{i \in \Lambda_n} \sum_{j \in \Lambda_n \smallsetminus \{i\}} e^{\nu(\mathcal{S}_{\{ i,j\}}^{\cap})} 
    \\&\qquad=
    \sum_{i \in \Lambda_n} \sum_{j \in K_{n,r}(i)}e^{\nu(\mathcal{S}_{\{ i,j\}}^{\cap})} 
    +
    \sum_{j \in \Lambda_n \smallsetminus (K_{n,r}(i) \cup \{ i \})}
    e^{\nu(\mathcal{S}_{\{ i,j\}}^{\cap})}.
  \end{align*}
  Since \( 0 \leq \nu(\mathcal{S}_{\{i,j\}}^\cap) \leq \nu(\mathcal{S}_i)=\nu(\mathcal{S}_0)<\infty\) for all \( i,j \in \Lambda_n, \)
  it follows that for any \( r >0,\) 
  
  \begin{align*}
    &1 \leq \lim_{n \to \infty} n^{-2d}\sum_{i,j \in \Lambda_n \colon i \neq j} e^{\nu(\mathcal{S}_{\{ i,j\}}^{\cap})} 
    =
    \lim_{n \to \infty} n^{-2d}\sum_{i \in \Lambda_n} \sum_{j \in \Lambda_n \smallsetminus ( K_{n,r}(i) \cup \{ i \})}
    e^{\nu(\mathcal{S}_{\{ i,j\}}^{\cap})} 
    \leq e^{r} .
  \end{align*} 
  Since \( r>0\) was arbitrary, we obtain~\eqref{eq: goal 2}.
\end{proof}

Finally, we present the proof of Lemma~\ref{lemma: variance and limit}.

\begin{proof}[Proof of Lemma~\ref{lemma: variance and limit}] 
  
  Fix \( n \geq 1 .\) 
  For $x \in \bZ^d$ let $\tau_x \colon \bZ^d \to \bZ^d$ be the shift of $\bZ^d$ by $x$, i.e.\ the map which for each \( y \in \bZ^d\) maps \( y \) to $\tau_x(y) = x+y$. 
  %
  Note that with this notation, we have 
    $ x \in \Lambda_{2n} \Leftrightarrow \tau_x (\Lambda_n) \cap \Lambda_n \neq \emptyset.$  
  Further, for $\delta \in (0,1)$, let
  \begin{equation}
    T^{n,\delta} \coloneqq \bigr\{ \Delta \subseteq \bZ^d \colon 
    \exists x \in \Lambda_{2n} \text{ such that } | \tau_x( \Lambda_n) \cap \Delta| \geq  \delta |\Lambda_n|\bigr\}.
  \end{equation} 
  In other words, a set \( \Delta \subseteq \bZ^d \) is in \( T^{n,\delta} \) if it has density at least \( \delta\) in some box \( \Lambda\) with side length \( 2n\) which intersects \( \Lambda_n.\)  

    For \( j \geq 1,\) let \( \mathcal{E}_{n,j}\) be the event that \( Y_{n,\delta} \geq j\) where \( Y_{n,\delta} \sim \mathrm{Poisson}(\nu|_{{T}^{n,\delta}}). \)

    We first show that  \( \limsup_{n\to \infty} \nu(T^{n,\delta}) < \infty.\) To this end, note first that for any \( i \in \Lambda_{3n},\) by translation invariance, we have
    \begin{equation}
        P(X_n(i) = 1 \mid \mathcal{E}_{n,j}) \geq  1 - \bigl( 1-\frac{\delta}{3^d} \Bigr)^j.
    \end{equation}
    If \( \limsup_{n\to \infty} \nu(T^{n,\delta}) = \infty,\) then \( \limsup_{n \to \infty} P(\mathcal{E}_{n,j}) = 1\) for any \( j \geq 1.\) By translation invariance, this implies that \( \limsup_{n \to \infty} P(\bar X_{3n} \geq c) = 1.\) Since this contradicts~\eqref{eq: limit assumption}, it follows that \( {\limsup_{n\to \infty} \nu(T^{n,\delta}) < \infty}.\)

  Now, let \( c \in (0,1) \) be such that~\eqref{eq: limit assumption} holds, and choose $k \geq 1$ such that 
  \begin{equation}\label{eq k}
  \Bigl(1-\frac{\delta}{3^d}\Bigr)^k < \frac{1-c}{2}.
  \end{equation}
  Then 
  \[
    \bP(\mathcal{E}_{n,k})= \sum_{\ell \ge k}^\infty
      \frac{e^{-\nu(\mathcal{T}^{n,\delta})} \nu(\mathcal{T}^{n,\delta})^\ell}{\ell!}.
  \]  
  We claim that for each $i \in \Lambda_n$ and any $k$ it holds that 
  \begin{equation}\label{eq: key bound for this lemma}
    \bP\bigl(X_n(i)= 0\mid \mathcal{E}_{n,k} \bigr)\leq  \Bigl(1-\frac{\delta}{3^d}\Bigr)^k  
  \end{equation}
  Before presenting the proof of this claim, we show how it implies the statement of the lemma. Particularly, applying Markov's inequality conditioned on $\mathcal{E}_{n,k}$, we get
  \begin{align*}
    &\bP\bigl(\bar X_{n} \ge c\mid \mathcal{E}_{n,k}\bigr) 
    =
    1-\bP\Bigl(\sum_{i\in \Lambda_{n}} \bigl(1-X(i)\bigr)  > (1-c)|\Lambda_n|\mid \mathcal{E}_{n,k}\Bigr)
    \\&\qquad\geq 1-\frac{\mathbb{E}\bigl[ \sum_{i\in \Lambda_{n}} \bigl(1-X(i)\bigr)\mid \mathcal{E}_{n,k} \bigr] } {(1-c) |\Lambda_n|}
    \ge 1-\frac{(1-c)|\Lambda_n|/2}{(1-c) |\Lambda_n|}=\frac{1}{2}.
  \end{align*}
  From this, it follows that
  \begin{equation*}    
    \bP(\bar X_{n} \ge c)\ge \frac{\bP(\mathcal{E}_{n,k})}{2}.
  \end{equation*}
  Using~\eqref{eq: limit assumption}, we therefore obtain \( \lim_{n \to \infty} \bP(\mathcal{E}_{n,k}) = 0 \). 
  Noting that 
  \[
  \bP(\mathcal{E}_{n,k})= \sum_{\ell \ge k}^\infty
      \frac{e^{-\nu(\mathcal{T}^{n,\delta})} \nu(\mathcal{T}^{n,\delta})^\ell}{\ell!} \geq \frac{e^{-\nu(\mathcal{T}^{n,\delta})} \nu(\mathcal{T}^{n,\delta})^k}{k!}
    \]
    and recalling that \( {\limsup_{n\to \infty} \nu(T^{n,\delta}) < \infty},\)
    it follows  that $ \lim_{n \to \infty} \nu(\mathcal{T}^{n,\delta})=0$. From this we conclude the proof since \( \mathcal{S}^{n,\delta} \subseteq \mathcal{T}^{n,\delta}. \)  
    
  It remains to show that \eqref{eq: key bound for this lemma} holds. For this,   
 for \( \Delta \in T^{n,\delta}, \) let 
  \begin{equation*} 
    \hat \Delta \coloneqq \bigr\{ i \in \Delta \colon 
    \exists x \in \Lambda_{2n} \text{ such that }  |\tau_x( \Lambda_n) \cap \Delta| \geq  \delta |\Lambda_n| \text{ and } i \in \tau_x( \Lambda_n) \bigr\}. 
  \end{equation*}
  Note that, by construction, it holds that $ \hat \Delta \subset \Delta$. Further, since \( \Delta \in T^{n,\delta},\) we know that there is \(x \in \Lambda_{2n}\) such that \(\tau_x( \Lambda_n) \cap \Lambda_n \neq \emptyset \) and \( |\tau_x( \Lambda_n) \cap \Delta| \geq  \delta |\Lambda_n| \). This implies that \( |\hat \Delta| \geq \delta |\Lambda_n|\).  
  
  Now, for any
  \( \Delta \in T^{n,\delta}, \) \( i \in \Lambda_{n},\) and \( j \in \hat \Delta ,\) we have \( \tau_{i-j} \Delta \in T^{n,\delta}\) and \( i \in \widehat{\tau_{i-j} \Delta}.\) Moreover, as we argue below, it holds that  
  \begin{equation}\label{eq: translation implication}
    i \in \widehat {\tau_{i-j} \Delta} \subseteq \tau_{i-j} \Delta \in T^{n,\delta}.
  \end{equation} 
  Since \( \nu \) is translation invariant, it follows from~\eqref{eq: translation implication} that, if we let \( \Delta \sim \nu|_{T^{n,\delta}}\), then 
  \begin{equation}\label{eq: eq and ineq} 
    \bP (i \in \hat \Delta) \geq \bP (j \in \hat \Delta).
  \end{equation} 
  From this, it follows that 
  \begin{align*}
    |\Lambda_{3n}| \bP(i \in \hat \Delta) = \sum_{j \in \Lambda_{3n}} \bP(i \in \hat \Delta) 
    \geq 
    \sum_{j \in \Lambda_{3n}} \bP(j \in \hat \Delta) = \mathbb{E}\bigl[|\hat \Delta \cap \Lambda_{3n}|\bigr].
  \end{align*}  
  Hence 
  \[
    \bP(i \in \Delta) \geq \bP(i \in \hat \Delta) \geq \frac{\mathbb{E}\bigl[|\hat \Delta \cap \Lambda_{3n}|\bigr]}{|\Lambda_{3n}|}.
  \]
  Since \( |\hat \Delta| = |\hat \Delta\cap \Lambda_{3n}| \geq \delta |\Lambda_n|\) by definition, it follows that
  \begin{equation}\label{eq: key bound for this lemma 1}
    \bP(i \in \Delta) \geq \frac{\delta |\Lambda_n|}{|\Lambda_{3n}|} = \frac{\delta}{3^d}.
  \end{equation}
  By this, the independence of the Poisson process and \eqref{eq k}, we conclude \eqref{eq: key bound for this lemma}.

It remains to show that \eqref{eq: translation implication} holds. 
 To this end, let \( \Delta \in T^{n,\delta}, \) \( i \in \Lambda_{n},\) and \( j \in \hat \Delta.\) 
  We first show that \( \tau_{i-j}\Delta \in T^{n,\delta}.\) Indeed, note that since \( j \in \hat \Delta ,\) there is \( x \in \Lambda_{2n}\) such that \(|\tau_x(\Lambda_n)\cap \Delta|\geq \delta|\Lambda_n|\) and \( j \in \tau_x(\Lambda_n). \) Fix one such \( x,\) and let \( y \coloneqq x+i-j.\)  
  Then
  \[
  \tau_y (\tau_x^{-1}j) = \tau_{x+i-j} \bigl( \tau_x^{-1} (j)\bigr) = \tau_{i-j} j.
  \]
  Since \( \tau_{i-j} j = i \in \Lambda_n\) by assumption, it follows that \( \tau_y(\Lambda_n) \cap \Lambda_n \neq \emptyset, \) and thus \( y \in \Lambda_{2n}.\)
  Moreover, we have
  \[ | \tau_y\Lambda_n \cap \tau_{i-j}\Delta| = | \tau_x\tau_{i-j}\Lambda_n \cap \tau_{i-j}\Delta| = | \tau_x \Lambda_n \cap \Delta| \geq \delta|\Lambda_n|.
  \] 
  This shows that \( \tau_{i-j}\Delta \in T^{n,\delta}.\)  
  We now argue that also \( i \in \widehat{\tau_{i-j} \Delta}.\) For this, note that since \( j \in \hat \Delta, \) we have 
  \[ i = \tau_{i-j} j\in \tau_{i-j} \hat \Delta \subseteq \tau_{i-j} \Delta ,\] 
  and since \( j \in \tau_x(\Lambda_n) \), we have
  \begin{equation*}
    i = \tau_{i-j} j \in \tau_{i-j} \tau_x(\Lambda_n) = \tau_y(\Lambda_n),
  \end{equation*}
  This shows that \( i \in \widehat{\tau_{i-j}\Delta },\) and thus completes the proof of~\eqref{eq: translation implication}.
  \end{proof}
 
\section{Mixing properties for the contact process}\label{sec:CP}

In this section, we present the proof of Theorem~\ref{theorem:cp2}. For this, we first recall the basic constructions of the contact process on $\bZ^d$.

As detailed in \cite[Chapter 1]{LiggettSIS1999}, the contact process with parameter $\lambda\in(0,\infty)$ can be specified in terms of it pre-generator $L_{\lambda} \colon \mathcal{C}(\mathbb{R}) \mapsto \mathcal{C}(\mathbb{R})$, where $\mathcal{C}(\mathbb{R})$ denotes the set of bounded and continuous functions $f \colon \Omega \rightarrow \mathbb{R}$ and 
$\Omega \coloneqq \{0,1\}^{\bZ^d}$. For the contact process, this is given by 
\begin{equation}\label{eq contact generator}
  L_{\lambda} f(\omega) \coloneqq \sum_{\substack{x \in \bZ^d \mathrlap{\colon}\\ \omega(x)=1}} \Bigl( \bigl[f(\omega^{x \leftarrow 0})-f(\omega)\bigr]+\lambda \sum_{y \sim x} \bigl[f(\omega^{y \leftarrow 1})-f(\omega) \bigr] \Bigr), \quad \omega \in \Omega.
\end{equation} 
Here, for \( \omega \in \Omega, \) $z \in \bZ^d,$ and $i \in \{0,1\}$, we write $\omega^{z \leftarrow i} \in \Omega$ for the configuration with $\omega^{z \leftarrow i}(z)=i$ and $\omega^{z \leftarrow i}(x)=\omega(x)$ for $x \in \mathbb{Z}^d \smallsetminus \{z\}$.  

The contact process can be constructed using terminology from percolation theory by the graphical construction~\cite[Chapter 1.1]{LiggettSIS1999}.
 We also recall this construction for the reader's convenience, as it will be helpful in the following arguments.

For each $x\in \bZ^d$ and each ordered pair $(x,y)$ of nearest neighbour vertices in \( \mathbb{Z}^d,\) let $(N_x)$ and $(N_{(x,y)})$ be independent Poisson processes on $\bR$ with rate $1$ and rate $\lambda$, respectively. 
An event of $N_x$ represents a potential "healing event" where the state at \( x \) at that time is set to \( 0 \), whereas an event of $N_{(x,y)}$ represents a potential ``infection event'', where the state at \( y \) will be set to \( 1 \) if the state at either \( x\) or \( y \) is \( 1 \).
For each $x,y \in \bZ^d$ and $s\leq t$, we say that $(x,s)$ is connected to $(y,t)$ by an \emph{active} path, written $(x,s) \rightarrow (y,t)$, if and only if there exists a path in $\bZ^d \times \bR$ starting at $(x,s)$ and ending at $(y,t)$ that goes forwards in time without hitting any healing event and that may cross to another vertex at the instance of an infection event in the prescribed direction of the ordered pair. 
That is, there exists a sequence $x=x_0,x_1,\dots,x_n=y$ in $\bZ^d$ and times $s=t_0<t_1<\dots<t_{n+1} = t$ such that, for each $i=0,\dots,n$, there are no healing events at $x_i$ within time $[t_i,t_{i+1}]$, but there is an infection event at $(x_i,x_{i+1})$ within the same time window. 

Now, denote by $(\eta_t)_{t \in [0,\infty)}$ the process on $\Omega$ given by \begin{equation}\label{def:GP_CP}
  \eta_t(x) \coloneqq \ind \bigl( \forall s < t \; \exists y \in \bZ^d \colon (y,s) \rightarrow (x,t) \bigr), \quad x \in \bZ^d,\, t \geq 0. 
\end{equation}
Then $\eta_0$ is distributed according to the upper invariant measure $\mu_{\lambda}$. Moreover, in distribution, $(\eta_t)_{t \in [0,\infty)}$ equals the same process as that defined via \eqref{eq contact generator} with initial distribution given $\mu_{\lambda}$. 
Further, by construction, the process $(\eta_t)$ is time-stationary so that $\eta_t \sim \mu_{\lambda}$ for any $t\geq 0$.  
In the following, we denote by $\bP_{\lambda}$ the distribution of \eqref{def:GP_CP} on the probability space on which the processes of the graphical construction introduced above are defined. Moreover, we write $o \in \bZ^d$ for the origin.  

\subsection{The upper invariant measure of the contact process is not spatially mixing}

In this subsection, we present a proof of Theorem~\ref{theorem:cp2}\ref{item: cp2b}. For this, we first provide an extension of~\cite[Proposition 2.1]{LiggettSteifSD2006} to the contact process on $\bZ^d.$

\begin{proposition}\label{lem:cp1b}
  Let $d\geq 1$ and $\lambda > \lambda_c$, and consider $X \sim \mu_{\lambda}$ on $\{0,1\}^{\bZ^d}.$ Then
  \begin{equation}\label{eq:lem:cp1111}
    \lim_{n \rightarrow \infty} \mu_{\lambda} \bigl( X(o)=0\mid X(\Lambda_n \setminus \{o\}) \equiv 0 \bigr) = 0. 
  \end{equation}
\end{proposition}

The proof of Proposition~\ref{lem:cp1b} is a direct extension of that of~\cite[Proposition 2.1]{LiggettSteifSD2006} from $d=1$ to general dimensions. It uses the description of the contact process by its pre-generator in \eqref{eq contact generator} and the particular property, since $\mu_{\lambda}$ is stationary, that $\int L_{\lambda} g \, d\mu_{\lambda}=0$ for any cylinder function $g$; see, e.g.,~\cite[Theorem B7]{LiggettSIS1999}.

\begin{proof}[Proof of Proposition~\ref{lem:cp1b}]
  Let $m \in \bN$, and let $g \colon \Omega \rightarrow \{0,1\}$ be the cylinder function given by 
  \[ 
    g(\omega) \coloneqq \ind \bigl( \sum_{i \in \Lambda_m} \omega(i) =0 \bigr). 
  \] 
  Further, let $\partial_e \Lambda_m$ denote the set of all ordered pairs $(x,y)$ such that $x \in \Lambda_m$ and $y \notin \Lambda_m$ with $x\sim y$. 
  Then, 
  \begin{align*}
    &\int L_\lambda g \, d\mu_\lambda 
    = 
    \int \sum_{\substack{x \in \bZ^d \mathrlap{\colon}\\ \omega(x)=1}} \Bigl( \bigl[g(\omega^{x \leftarrow 0})-g(\omega)\bigr]+\lambda \sum_{y \sim x} \bigl[g(\omega^{y \leftarrow 1})-g(\omega) \bigr] \Bigr) \, d\mu_\lambda
    \\&\qquad=
    \int \sum_{\substack{x \in \Lambda_m \mathrlap{\colon}\\ \omega(x)=1}} \Bigl( \bigl[g(\omega^{x \leftarrow 0}) \bigr]-\lambda \sum_{\substack{x \in \bZ^d \smallsetminus \Lambda_m \mathrlap{\colon}\\ \omega(x)=1}} \sum_{\substack{ y \in \Lambda_m \mathrlap{\colon}\\ y \sim x} } g(\omega) \Bigr) \, d\mu_\lambda
    \\&\qquad=
    \sum_{x \in \Lambda_m} \mu_{\lambda} \bigl( X(x)=1,\, X(\Lambda_m \setminus \{x\}) \equiv 0 \bigr)
    -\lambda \sum_{(x,y) \in \partial_e \Lambda_m} \mu_{\lambda} \bigl(X(y)=1,\, X(\Lambda_m) \equiv 0 \bigr).
  \end{align*} 
  Since $\int L_{\lambda} g \, d\mu_{\lambda}=0$, it follows that
  \begin{align}
    &\sum_{x \in \Lambda_m} \mu_{\lambda} \bigl( X(x)=1,\, X(\Lambda_m \setminus \{x\}) \equiv 0 \bigr)
    =\lambda \sum_{(x,y) \in \partial_e \Lambda_m} \mu_{\lambda} \bigl(X(y)=1,\, X(\Lambda_m ) \equiv 0 \bigr).
  \end{align}
  Dividing by $\mu_{\lambda} \bigl( X(\Lambda_m) \equiv 0 \bigr)$ on both sides, we see that 
  \begin{align}
    \sum_{x \in \Lambda_m} \frac{\mu_{\lambda} \bigl( X(x)=1 \mid X(\Lambda_m \setminus \{x\}) \equiv 0 \bigr)}{\mu_{\lambda} \bigl( X(x)=0 \mid X(\Lambda_m \setminus \{x\}) \equiv 0 \bigr)} 
     \leq \lambda \left|\partial_e \Lambda_m\right|.  
  \end{align} 
  Now note that by the d-FKG property, we have 
  \begin{align}
    \mu_{\lambda} \bigl( X(x)=0 \mid X(\Lambda_m \setminus \{x\}) \equiv 0 \bigr) \geq \mu_{\lambda} \bigl( X(x)=0 \bigr).
  \end{align}
  Also by the d-FKG property, for any $x \in \Lambda_m$ and any box \( \Lambda \supseteq \Lambda_m , \) we have that%
  \begin{align} 
    &\mu_{\lambda} \pigl( X(x)=1\mid X\bigl(\Lambda \setminus \{x\} \bigr) \equiv 0 \pigr) \leq \mu_{\lambda} \bigl( X(x)=1 | X(\Lambda_m \setminus \{x\}) \equiv 0 \bigr).
  \end{align}  
  Therefore, by translation invariance of the model, it follows that
  \begin{equation*}
    \left| \Lambda_m \right| \mu_{\lambda} \pigl( X(o)=1\mid X\bigl(\Lambda_{3m} \setminus \{o\} \bigr) \equiv 0 \pigr) \leq  \lambda \left| \partial_e \Lambda_m \right|\mu_{\lambda} \bigl( X(o)=0\bigr)
  \end{equation*}  
  Since \( \lim_{m \rightarrow \infty} { \left|\partial_e \Lambda_m \right|}/{ \left|\Lambda_m \right|}=0, \) letting $m\rightarrow \infty,$ we obtain \eqref{eq:lem:cp1111}. 
\end{proof}

We next show how to leverage Proposition~\ref{lem:cp1b} in order to prove Theorem~\ref{theorem:cp2}\ref{item: cp2b}.

\begin{proof}[Proof of Theorem~\ref{theorem:cp2}\ref{item: cp2b}]
  We prove Theorem~\ref{theorem:cp2}\ref{item: cp2b} via an inductive argument. Our induction hypothesis is that, for some \( k \geq 1\) and any set $\Delta_k\subseteq \bZ^d$ of cardinality $k,$ it holds that
  \begin{equation}
    \mu_{\lambda} \bigl( X(\Delta_k) \equiv 0 \mid X(\Delta_k^c) \equiv 0 \bigr) = 1.
  \end{equation}
  By Proposition~\ref{lem:cp1b}, we know this holds for $k=1$.

  Now, assume the induction hypothesis holds for all sets of cardinality $k$ and let $\Delta_{k+1} = \Delta_k \cup \{x\} \subset \bZ^d$ be a set with $k+1$ elements. 
  Then, for any non-trivial partition $\Delta_{k+1}=\Delta^{(0)} \cup \Delta^{(1)}$, we have that
  \begin{align*}
    &\mu_{\lambda} \bigl(X(\Delta^{(1)}) \equiv 1 ,\, X(\Delta^{(0)}) \equiv 0 | X(\Delta_{k+1}^c) \equiv 0\bigr)
    \\&\qquad =  \mu_{\lambda}(X(\Delta^{(1)}) \equiv 1 \mid X(\Delta^{(0)} \cup  \Delta_{k+1}^c) \equiv 0 ) 
     \, \mu_{\lambda}( X(\Delta^{(0)} ) \equiv 0  \mid X(\Delta_{k+1}^c) \equiv 0 \bigr).
  \end{align*}
  Since the first term on the right-hand side equals zero by the induction hypothesis, it follows that on the event \( X(\Delta_{k+1}^c)\equiv 0 ,\) \( X \) concentrates on either having all $1$'s or all $0$'s $\text{ on } \Delta_{k+1}$. %
  In particular, it follows that 
  \begin{equation}\label{eq:absurd2}
    \mu_{\lambda} \bigl(X(\Delta_{k+1}) \equiv 1 \mid X(\Delta_{k+1}^c) \equiv 0 \bigr) = \mu_{\lambda}\bigl(X(x)= 1 \mid X(\Delta_{k+1}^c) \equiv 0  \bigr). 
  \end{equation}
  Using Lemma~\ref{lemma: old claim}, stated below, it follows that this can only be true if both sides are equal to zero.
  Hence, we have that 
  \begin{equation*} 
    \mu_{\lambda} \bigl( X(\Delta_{k+1} ) \not\equiv 0 \mid X(\Delta_{k+1}^c) \equiv 0 \bigr) = 0.
  \end{equation*}
  From this the desired conclusion immediately follows. 
\end{proof}

\begin{lemma}\label{lemma: old claim}
  In the setting of the proof of Theorem~\ref{theorem:cp2}\ref{item: cp2b}, there is \( a \in (0,1) \) such that 
  \begin{equation}\label{eq: claimeq}
    \mu_{\lambda} \bigl( X(\Delta_k) \equiv 1 \mid X(\Delta_{k+1}^c) \equiv 0 \bigr) \leq a \mu_{\lambda} \bigl( X(x)=1 \mid X(\Delta_{k+1}^c) \equiv 0 \bigr).
  \end{equation}
\end{lemma}

To prove Lemma~\ref{lemma: old claim}, we will use the graphical representation of the contact process.

\begin{proof}[Proof of Lemma~\ref{lemma: old claim}]
  Recall that \( \Delta_{k+1} = \Delta_k \cup \{ x \}. \) Note that, by~\eqref{eq:absurd2}, the inequality in~\eqref{eq: claimeq} trivially holds if 
  \begin{equation}
    \bP_{\lambda} \bigl(X(x) = 1 \mid X (\Delta_{k+1}^c)\equiv 0 \bigr)=0.
  \end{equation}   
  We, therefore, assume that this quantity is strictly positive. Then, on the induction hypothesis made in the proof of Theorem~\ref{theorem:cp2}\ref{item: cp2b}, we have that
  \begin{align*}
    & \mu_{\lambda} \bigl( X(\Delta_{k}) \equiv 1 \mid X( \Delta_{k+1}^c ) \equiv 0 \bigr)
    \\ & \quad= \mu_{\lambda} \bigl( X(\Delta_{k}) \equiv 1 \mid X(x)=1, X( \Delta_{k+1}^c ) \equiv 0 \bigr) \mu_{\lambda} \pigl( X(x)=1 \mid X\bigl( \Delta_{k+1}^c \bigr) \equiv 0 \pigr). 
  \end{align*} 
  We will argue that $\mu_{\lambda} \bigl( X(\Delta_{k}) \equiv 1 \mid X(x)=1, X( \Delta_{k+1}^c ) \equiv 0 \bigr)<1$ from which the claim of the lemma immediately follows. For this, we will use the graphical construction of the contact process $(\eta_t)$ as given by \eqref{def:GP_CP}.  
  To this end, for $\delta>0$, let $B_{\delta}$ be the event that there is an infection event from or to a vertex in $\Delta_{k+1}$ within the time interval $[0,\delta]$. 
  Then, since $B_{\delta}$ is an increasing event,  
  and $\{\eta_{\delta}( \Delta_{k+1}^c ) \equiv 0\}$ is a decreasing event 
  (with respect to the percolation substructure obtained from the graphical construction 
  ), using that the contact process is positively associated (in space-time), we have that 
  \begin{equation}\label{eq: numerator}
    \bP_{\lambda} \bigl( B_{\delta} \mid \eta_{\delta}(\Delta_{k+1}^c ) \equiv 0 \bigr) \leq \bP_{\lambda}(B_{\delta} ).
  \end{equation}
  Next, by definition, we have
  \begin{align}\label{eq:lemcp555help}
    \bP_{\lambda} \bigl(B_{\delta} \mid \eta_{\delta}(x) = 1 ,\, \eta_{\delta}( \Delta_{k+1}^c ) \equiv 0  \bigr) = &\frac{\bP_{\lambda} \bigl(B_{\delta}, \eta_{\delta}(x) =1 \mid \eta_{\delta}( \Delta_{k+1}^c) \equiv 0 \bigr) }{\bP_{\lambda} \bigl(\eta_{\delta}(x) = 1 \mid \eta_{\delta}( \Delta_{k+1}^c ) \equiv 0 \bigr)}.
  \end{align}
  Since \( \eta_\delta\sim \mu_{\lambda}\) for any $\delta\geq0$, the denominator on the right-hand side of~\eqref{eq:lemcp555help} does not depend on \( \delta.\) Therefore, since \( \lim_{\delta \to 0} \bP_{\lambda}(B_{\delta} )= 0,\) using~\eqref{eq: numerator}, it follows that 
  \begin{equation}
    \lim_{\delta \to 0} \bP_{\lambda} \bigl(B_{\delta} \mid \eta_{\delta}(x)= 1,\, \eta_{\delta}( \Delta_{k+1}^c) \equiv 0 \bigr) = 0,
  \end{equation}
  and hence
  \begin{equation}
    q:= \bP_{\lambda} \bigl(B_{\delta}^c \mid \eta_{\delta}(x)= 1,\, \eta_{\delta}( \Delta_{k+1}^c) \equiv 0 \bigr)
  \end{equation}
  is strictly positive for all $\delta$ sufficiently small. 

  Now note that the event 
  \begin{equation*}
    \mathcal{E} \coloneqq \bigl\{ B_{\delta}^c,\, \eta_{\delta}(x)= 1 ,\, \eta_{\delta}( \Delta_{k+1}^c) \equiv 0 \bigr\}
  \end{equation*} 
  does not reveal any information about the healing events on $\Delta_k$ within the time interval $[0,\delta]$. 
  Moreover, with a strictly positive probability, say $p=p(\delta)>0$, the event $\mathcal{R}$ that there is a healing event at one of the vertices of $\Delta_k$ within this time interval occurs.  
  Hence, in the event that both $\mathcal{R}$ and $\mathcal{E}$ occur, we necessarily have that $\eta_\delta(\Delta_k) \not\equiv 1$. 
  Consequently, for $\delta>0$ sufficiently small, it holds that
   \begin{align*}
   \bP_{\lambda} \bigl(\eta_{\delta}(\Delta_k) \equiv 1 \mid \eta_{\delta}(x)=1, \eta_{\delta}( \Delta_{k+1}^c ) \equiv 0  \bigr) 
   < 1 - qp < 1.
      \end{align*} 
     From this, using that $\eta_{\delta} \sim \mu_{\lambda}$, 
    we conclude the proof. 
\end{proof}

\subsection{The upper invariant measure is directional mixing}

This subsection is devoted to the proof of Theorem~\ref{theorem:cp2}\ref{item: cp2a}. Our motivation for this statement stems from~\cite[Theorem 4.1 and Corollary 4.1]{LiggettSteifSD2006} which says that for the contact process on $\bZ^d$ with $\lambda>\lambda_c$ there is a $\rho= \rho(\lambda) >0$ such that, for any any $y=(y_1,\dots,y_d) \in \bZ^d$ and any disjoint finite sets $A,B \subset \bZ_{<y}^d$, it holds that
\begin{equation}\label{eq:domiCP_LS}
\mu_{\lambda} (X(y)=1 \mid X \equiv 0 \text{ on } A, X \equiv 1 \text{ on } B ) \geq \rho(\lambda),
\end{equation}
where $\bZ_{<y}^d$ denotes the set of $x=(x_1,\dots,x_d)\in \bZ^d$ such that $x_1<y_1$ or $x_i=y_i$ for $i=1,\dots,k$ and $x_{k+1}<y_{k+1}$ for some $k =1,\dots,d-1$.

In the terminology of the graphical construction, \eqref{eq:domiCP_LS} gives that there with positive probability is an infinite active path ending at the origin at time $0$, regardless of whether this happens for any point "to the left of" (with respect to the lexicographic ordering on $\bZ^d$) the origin. What makes \eqref{eq:domiCP_LS} particularly powerful is that this holds in a conditional sense and regardless of how unlikely the conditional event is.

Our proof of Theorem~\ref{theorem:cp2}\ref{item: cp2a} uses the inequality \eqref{eq:domiCP_LS} as an essential input. 
We first provide the proof of Theorem~\ref{theorem:cp2} in the case $d=1$, where we can give a short (and perhaps more elegant) argument. 
 
\begin{proof}[Proof of Theorem~\ref{theorem:cp2}\ref{item: cp2a} when $d=1$]
  By translation invariance, it suffices to show that 
  \begin{equation}\label{eq:lem:cp111111111111 d1}
    \lim_{n \rightarrow \infty}  \mu_{\lambda}\pigl( X(0) = 0 \mid X\bigl( (-\infty, -n) \bigr) \equiv 0 \pigr) 
     = \mu_{\lambda}\bigl(X(0)=0 \bigr).
  \end{equation}
  To this end, let 
  \[
    \rho \coloneqq \lim_{n \rightarrow \infty} \mu_{\lambda}\pigl(X(0)=1 | X\bigl( (-n,-1] \bigr) \equiv 0 \pigr).
  \]
  By~\eqref{eq:domiCP_LS}, \( \rho >0\) whenever $\lambda>\lambda_c$. 
  We will argue that, for any $n\geq 1$, 
  \begin{equation}\label{eq:convergenceCP}
    \mu_{\lambda}\bigl(X(0)=1\bigr) - \mu_{\lambda}\pigl(X(0)=1 \mid X\bigl( (-\infty-n] \bigr) \equiv 0 \pigr) \leq (1-\rho)^n.
  \end{equation}
  Note that~\eqref{eq:lem:cp111111111111 d1}  
  immediately follows from~\eqref{eq:convergenceCP} since, by the d-FKG property, we have 
  \[
     \mu_{\lambda}\pigl(X(0)=1 \mid X\bigl((-\infty, -n]\bigr) \equiv 0  \pigr) \leq \mu_{\lambda}\bigl( X(0)=1\bigr).
  \]
To see that~\eqref{eq:convergenceCP} holds, for $i > -n$, consider the event  
  \begin{equation}
    \cE_i \coloneqq \{ X\bigl((-n,i)\bigr) \equiv 0 \quad\text{and}\quad X(i)=1 \}.
  \end{equation} 
  Then, we can write 
  \begin{equation}\begin{split}\label{eq: expr as sum}
    &\mu_{\lambda}\pigl(X(0)=1 \mid X\bigl( (-\infty,-n] \bigr) \equiv 0\pigr) 
    \\ &\qquad= \sum_{i=-(n-1)}^0 \mu_{\lambda}\pigl(X(0)=1 ,\, \mathcal{E}_i \mid X\bigl( (-\infty,-n] \bigr) \equiv 0 \pigr). 
  \end{split}\end{equation}
  We next describe how to control the terms within the sum of \eqref{eq: expr as sum}. For this, 
  by~\cite[Theorem 2]{BergHaggstromKahnKonno2006} (and the remark immediately following its statement), 
  for any ${i \in \{ -(n-1),\dots, -1\}}$ and any $m \geq n,$ we have 
  \begin{align*}
    &\mu_{\lambda} \pigl( X(0)=1,\, X\bigl( (-m,i) \bigr) \equiv 0  \mid X(i) = 1 \pigr) 
    \\&\qquad\geq 
    \mu_{\lambda} \bigl( X(0)=1  \mid X(i) = 1 \bigr) 
    \mu_{\lambda} \pigl( X\bigl( (-m,i) \bigr) \equiv 0   \mid X(i) = 1 \pigr).
  \end{align*}
  Dividing by $\mu_{\lambda} \bigl( X\bigl( (-m,i) \bigr) \equiv 0  \mid X(i) = 1 \bigr)$ on both sides yields 
  \begin{equation}
    \mu_{\lambda} \left( X(0)=1 \mid X\bigl( (-m,i) \bigr) \equiv 0 ,\, X(i) = 1 \right) \geq \mu_{\lambda} \bigl( X(0)=1  \mid X(i) = 1 \bigr).
  \end{equation}
  Taking the limit \( m \to \infty, \) we obtain
  \begin{equation}
    \mu_{\lambda} \pigl( X(0)=1 \mid X \bigl( (-\infty,-n] \bigr) \equiv 0 ,\, \mathcal{E}_{i} \pigr) 
    \geq 
    \mu_{\lambda} \bigl( X(0)=1 \mid X(i) = 1 \pigr)\geq \mu_{\lambda} \bigl( X(0)=1  \bigr).
  \end{equation}
  Since
  \begin{align*}
    &\mu_{\lambda} \pigl( X(0)=1 ,\, \mathcal{E}_{i} \mid X \bigl( (-\infty,-n] \bigr) \equiv 0 \pigr) 
    \\&\qquad=
    \mu_{\lambda} \pigl( X(0)=1  \mid X \bigl( (-\infty,-n] \bigr) \equiv 0 ,\, \mathcal{E}_{i} \pigr) 
    \mu_{\lambda} \pigl( \mathcal{E}_{i} \mid X \bigl( (-\infty,-n) \bigr) \equiv 0 \pigr) ,
  \end{align*}
it follows that 
  \begin{align*}
    & \mu_{\lambda}\pigl(X(0)=1 ,\, \mathcal{E}_i \mid X\bigl( (-\infty,-n] \bigr) \equiv 0 \pigr) 
    \\&\qquad \geq \mu_{\lambda}\bigl(X(0)=1 \bigr) \mu_{\lambda}\pigl( \mathcal{E}_i \mid X\bigl( (-\infty,-n] \bigr) \equiv 0 \pigr) .
  \end{align*}
  Inserting this into~\eqref{eq: expr as sum}, 
  we obtain
  \begin{align*}
    &\mu_{\lambda}\pigl(X(0)=1 \mid X\bigl( (-\infty,-n] \bigr) \equiv 0\pigr) 
    \\&\qquad \geq \mu_{\lambda}\bigl(X(0)=1 \bigr) \sum_{i=-(n-1)}^{-1} \mu_{\lambda}\pigl( \mathcal{E}_i \mid X\bigl( (-\infty,-n] \bigr) \equiv 0 \pigr).
  \end{align*} 
  Therefore, by construction, we have
  \begin{align}
    &\sum_{i=-(n-1)}^{0} \mu_{\lambda}\pigl( \mathcal{E}_i \mid X\bigl( (-\infty,-n] \bigr) \equiv 0 \pigr)  
    = 1-(1-\rho)^{n},
  \end{align}
and from which we conclude~\eqref{eq:convergenceCP}.
\end{proof}

  We now turn to the proof of Theorem~\ref{theorem:cp2}\ref{item: cp2a} for general \( d \geq 1\), which uses the following extension of~\cite[Corollary 4.1]{LiggettSteifSD2006} as the key technical lemma.

\begin{lemma}\label{lem:CPdomi}
  Let $\lambda> \lambda_c$. Then there is $\rho=\rho(\lambda)>0$ such that, for any $L\geq 0$, 
  the distribution
  \begin{equation*}
    \lim_{n \rightarrow \infty}    \bP_{\lambda} \pigl( \eta_0 \in \cdot \mid \eta_L\bigl(\bZ^{d-1} \times (-\infty, -n)\bigr) \equiv 0 \pigr)
  \end{equation*} 
  stochastically dominates a Bernoulli product measure with success probability $\rho$. 
\end{lemma}

\begin{proof}
 Fix $L>0$, let $\phi>0,$ and consider a partition $(\Delta^{(L,\phi)}(i))_{i \in \bZ^d}$ of $\bZ^d \times [0,L]$ where, for \( {i = (i_1,\dots, i_d) \in \mathbb{Z}^d} \) and writing $x=(x_1,\dots,x_d)$, we let 
  \begin{equation}
    \Delta^{(L,\phi)}(i) \coloneqq  \bigl\{ (x,t) \in \bZ^d \times [0,L] \colon x_j \in [i_j - \phi t, i_j - \phi t +1),\, j=1,\dots,d \bigr\}. 
  \end{equation}

 Consider the random variables  
  $(Z^{(L,\phi)}(i))_{i \in \bZ^d}$ given by
  \begin{equation}
    Z^{(L,\phi)}(i) \coloneqq \max \bigl\{ \eta_t(x) \colon (x,t) \in \Delta^{(L,\phi)}(i) \bigr\}, \quad i \in \bZ^d.
  \end{equation} 
  Since the d-FKG property is preserved under taking maximum and the random variables \( (\eta_t (x)) \) have the d-FKG property, the random variables $(Z^{(L,\phi)}(i))$ also have the d-FKG property; see~\cite[Lemma 2.1 and Lemma 2.2]{BergBethuelsen2018}.  
  Further, note that 
  \[
    \{Z^{(L,\phi)}(\Lambda_n) \equiv 0 \} \subset \{\eta_0( \Lambda_n) \equiv 0 \}.
  \]
   Therefore, recalling that $\mu_{\lambda}$ stochastically dominates $\pi_{\rho}$ for some $\rho>0$, as concluded in~\cite[Corollary 4.1]{LiggettSteifSD2006},  
  we obtain from~\cite[Theorem 4.1]{LiggettSteifSD2006} that  
\begin{equation}\label{eq:ZLbound} 
    \bP_{\lambda} \bigl(Z^{(L,\phi)}(o) =1 \mid Z^{(L,\phi)}(\bZ^d_{<o}) \equiv 0 \bigr) \geq \rho.
  \end{equation}

  Below, we argue that this implies that 
  \begin{equation}\label{eq:keyForDomi}
  \lim_{\phi \rightarrow \infty} \bP_{\lambda} \bigl(\eta_0(o) =1 \mid Z^{(L,\phi)}(\bZ^d_{<o}) \equiv 0 \bigr) 
  \geq \rho.
  \end{equation} 
  For this, first note 
   that for the event $\{Z^{(L,\phi)}(o) = 1\}$ in \eqref{eq:ZLbound} to hold, either the event $ \mathcal{E}_1 \coloneqq \{ \eta_0(o)=1 \}$ holds, or the event \( \mathcal{E}_2 \) that $ \{ \eta_0(o)=0 \}$ and there is an infinite active path ending at some other space-time location of $\Delta^{(L,\phi)}(o)$ holds. We will argue that the probability of the latter event decays to $0$ when $\phi$ is made large. 

  For $A,B \subset \bZ^d \times \bR$, write $A\rightarrow B$ for the event that an active path exists starting in $A$ and ending in $B$. 
  Then, for any \( R>0 \) we have that 
  \begin{align}
    &\bP_{\lambda} \bigl(\mathcal{E}_2 \mid Z^{(L,\phi)}(\bZ^d_{<o}) \equiv 0 \bigr) \label{eq: zeroeth term}
    \\&\quad\leq
    \bP_{\lambda}\pigl( \bigl(\bZ^d\setminus \{0\}) \times \{0\} \rightarrow \Delta^{(L,\phi)}(o) \bigr) \cap \eta_0(o)=0 \mid Z^{(L,\phi)}(\bZ^d_{<o}) \equiv 0 \pigr)
    \\ &\quad \leq \bP_{\lambda}\pigl((\bZ^d\setminus \{0\}) \times \{0\} \rightarrow \bigl(\Delta^{(L,\phi)}(o) \cap (\Lambda_R\times [0,L]) \bigr)\mid Z^{(L,\phi)}(\bZ^d_{<o}) \equiv 0 \pigr)\label{eq: first term}
    \\ &\qquad\qquad+ \bP_{\lambda}\pigl(
    (\bZ^{d}\setminus \bZ^d_{<o}) \times\{0\} \rightarrow \bigl(\Delta^{(L,\phi)}(o) \cap (\bZ^d \setminus \Lambda_R ) \times [0,L] \bigr)\pigr). \label{eq: second term} 
\end{align}  
  To bound the second term, we first recall from~\cite[Theorem 1.4]{GaretMarchandLDP2014} that there is a constants $\mu>0$ so that, for any $\epsilon>0$ there are $C,c>0$, depending only on $\lambda$ and $d$, such that,
  \begin{equation}  
    \label{eq:LDB_GM14}
    \bP_{\lambda} \pigl( \inf_{t>0} \bigl\{ (o,0) \rightarrow (x,t) \bigr\} \leq (1-\epsilon) \mu \norm{x} \pigr) \leq Ce^{-c\norm{x}}.
  \end{equation} 
  Further, note that any point $(x,t) \in \Delta^{(L,\phi)}(o) \setminus (\Lambda_R \times [0,L])$ is such that the spatial coordinate $x \in \bZ^d$ is at least at distance $R$ from $\bZ^d \setminus \bZ^d_{<o}$.
  Therefore, with $\epsilon=1/2$, whenever $R\geq 2L/\mu$, we have that 
  \begin{align}
    & \bP_{\lambda}\pigl( (\bZ^d \setminus \bZ^d_{<o}) \times \{0\} \rightarrow \bigl(\Delta^{(L,\phi)}(o) \cap (\bZ^d \setminus \Lambda_R ) \times [0,L] \bigr) \pigr)
    \\ \leq & \sum_{x \in \bZ^d \setminus \bZ^d_{<o}} \quad \sum_{\substack{y \in \mathbb{Z}^d \smallsetminus\Lambda_R \colon \\ \{y\}\times [0,L] \cap \Delta^{(L,\phi)}(o) \neq \emptyset}}
     \bP_{\lambda}\pigl( \inf_{t>0} \bigl\{ (x,0) \rightarrow (y,t) \bigr\} \leq L \pigr)
     \\ \leq & \sum_{x \in \bZ^d \setminus \bZ^d_{<o}} \quad \sum_{\substack{y \in \mathbb{Z}^d \smallsetminus\Lambda_R \colon \\ (\{y\}\times [0,L]) \cap \Delta^{(L,\phi)}(o) \neq \emptyset}} Ce^{-c\norm{y-x}} 
    \leq C_1e^{-c_1R},
  \end{align}
  for some constants $C_1,c_1 \in (0,\infty)$ only depending on $\lambda, d$ and $L$.
  Thus, the probability in~\eqref{eq: second term} can be made arbitrarily close to $0$ by tuning $R$ large. 
  Next, we claim that the probability in~\eqref{eq: first term} can be made arbitrarily close to zero by choosing $\phi$ large.  
  To see this, note that given a vertex \( x,\) the probability that there is an infection arrow from \( x\) in the time-window $[0,\phi^{-1}]$ is given by $1-e^{-\frac{2d\lambda}{\phi}}$. Hence, by translation invariance and the d-FKG property of the process, and again applying a standard union bound, it follows that
  \begin{align}
    &\bP_{\lambda}\pigl((\bZ^d\setminus \{0\}) \times \{0\} \rightarrow \bigl(\Delta^{(L,\phi)}(o) \cap (\Lambda_R\times [0,L]) \bigr) \mid Z^{(L,\phi)}(\bZ^d_{<o}) \equiv 0 \pigr) 
    \\ &\qquad \leq 
    \bP_{\lambda}\pigl( \exists (y,t) \in \bigl(\Delta^{(L,\phi)}(o) \cap (\Lambda_R\times [0,L]) \bigr) \colon N_{x,y}=t \text{ for some } x\sim y \pigr) 
    \\ &\qquad \leq |\Lambda_R| (1-e^{-\frac{2d\lambda}{\phi}}) .
  \end{align}
  By combining the upper bounds for~\eqref{eq: first term} and~\eqref{eq: second term}, it follows that~\eqref{eq: zeroeth term} can be made arbitrarily small by taking \( R \) and \( \phi \) large. 
  This implies, in particular that  
  \begin{align}
    \bP_{\lambda} \bigl(Z^{(L,\phi)}(o)=1 \mid Z^{(L,\phi)}(\bZ^d_{<o}) \equiv 0 \bigr) - \bP_{\lambda} \bigl( \mathcal{E}_1 \mid Z^{(L,\phi)}(\bZ^d_{<o}) \equiv 0 \bigr) 
  \end{align} 
  can be made arbitrarily small, by taking \( R \) and \( \phi \) large, and hence~\eqref{eq:ZLbound} implies~\eqref{eq:keyForDomi}.

  The claim of the lemma now follows by utilizing the d-FKG property and \eqref{eq:keyForDomi} via sequential coupling in a similar manner as in the proof of~\cite[Theorem 4.1]{LiggettSteifSD2006}, which we now explain. For this, consider the lexicographic order \( < \) on $\bZ^d$ 
  where $(x_1,\dots,x_d) <(y_1,\dots,y_d)$ if $x_1<y_1$ or, for some $k=1,\dots,d-1$, it holds that $x_i=y_i$ and $x_{k+1}<y_{k+1}$.  
  Any (finite) set $
  \Delta =\{z_1,\dots,z_m\} \subset \bZ^{d}$ can then be ordered so that, for every $i \in \{1,\dots,m\}$, we have 
  \[\{z_1,\dots,z_{i-1}\} \times \{0\} \subset \cup_{y<z_i}\Delta_y^{(L,\phi)} 
  .\] 
 For every such $i \in \{ 1,2, \dots, m \}$ and $\sigma \in \{0,1\}^{i-1}$, utilising the d-FKG property and setting $\phi=n/L$, we have that
  \begin{align}
    &\bP_{\lambda} \bigl( \eta_0(z_i)=1 \mid \eta_0(z_j)=\sigma_j, j=1,\dots,i-1 \text{ and } \eta_L\bigl(\bZ^{d-1} \times (-\infty, -n)\bigr) \equiv 0 \bigr)
    \\ &\qquad\geq \bP_{\lambda} \bigl( \eta_0(z_i)=1 \mid Z^{(L,n/L)}(\bZ^d_{< z_i}) \equiv 0 \bigr)
  \end{align} 
  \color{black}
  Therefore, letting $n\rightarrow \infty$, we can apply the bound in \eqref{eq:keyForDomi} and translation invariance of the process to obtain 
  \begin{equation}
    \lim_{n \rightarrow \infty}  \bP_{\lambda} \bigl( \eta_0(z_i)=1 \mid \eta_0(z_j)=\sigma_j, j=1,\dots,i-1 \text{ and } \eta_L\bigl(\bZ^{d-1} \times (-\infty, -n)\bigr) \equiv 0 \bigr) \geq \rho.
  \end{equation} 
  This yields the domination over a product Bernoulli distribution since any increasing event involving only the values attained on $\Delta$ can be decomposed into events of the form $\{ \eta_0(\tilde{\Delta})\equiv 1\}$ with $\tilde{\Delta} = \{\tilde{z}_1,\dots,\tilde{z}_l \} \subset \Delta$ and using that 
  \begin{align}
    &\bP_{\lambda} \pigl( \eta_0(\Delta) \equiv 1 \mid \eta_L\bigl(\bZ^{d-1} \times (-\infty, -n)\bigr) \equiv 0 \pigr)
    \\&\qquad= \prod_{i=1}^l \bP_{\lambda} \pigl( \eta_0(z_i)=1 \mid \eta_0(z_j)=1, j=1,\dots,i-1 \text{ and } \eta_L\bigl(\bZ^{d-1} \times (-\infty, -n)\bigr) \equiv 0 \pigr). 
  \end{align} 
\end{proof}

We now present the proof of Theorem~\ref{theorem:cp2}\ref{item: cp2a} for the general case when $d\geq 1.$ 
\begin{proof}[Proof of Theorem~\ref{theorem:cp2}\ref{item: cp2a}, $d\geq 1$]
  We aim to prove that,  for any $x \in \bZ^d$,
  \begin{equation}\label{eq: goal in ast proof}
    \lim_{n \rightarrow \infty}  \mu_{\lambda}\pigl( X(x) = 1 \mid X\bigl(\bZ^{d-1} \times (-\infty, -n) \bigr) \equiv 0 \pigr) = \mu_{\lambda}\bigl(X(x)=1 \bigr).
  \end{equation} 
  By translation invariance and the d-FKG property, it is sufficient to show that~\eqref{eq: goal in ast proof} holds for $x=o$. Moreover, by the d-FKG property, we know that 
  \begin{equation}
   \mu_{\lambda}\pigl( X(o) = 1 \mid X\bigl(\bZ^{d-1} \times (-\infty, -n) \bigr) \equiv 0 \pigr) \leq \mu_{\lambda}\bigl(X(o)=1 \bigr)
  \end{equation}
  for any $n \in \bN$. Therefore, to conclude \eqref{eq: goal in ast proof}, we will argue that 
  \begin{equation}\label{eq: goal in last proof}
    \lim_{n \rightarrow \infty}  \mu_{\lambda}\pigl( X(o) = 1 \mid X\bigl(\bZ^{d-1} \times (-\infty, -n) \bigr) \equiv 0 \pigr) \geq \mu_{\lambda}\bigl(X(o)=1 \bigr).
  \end{equation} 
  For this, recall the graphical construction of the contact process. For $L>0$ and $R\in \bN$, consider the (random) sets 
  \begin{align}
    D_{L,R} \coloneqq \{ x \in \Lambda_R \colon  (x,0) \rightarrow (o,L) \},\quad F_{R }\coloneqq \{ x \in \Lambda_R \colon  \eta_0(x)=1 \}.
  \end{align}
  For \( n \geq 1,\) using the time-stationarity of $(\eta_t)$ and the d-FKG property, we have 
  \begin{equation}\label{eq: long inequality}\begin{split}
    &\mu_{\lambda}\pigl( X(o) = 1 \mid X\bigl(\bZ^{d-1} \times (-\infty, -n) \bigr) \equiv 0 \pigr) 
    \\ &\qquad=\bP_{\lambda}(\eta_L(o)=1 \mid \eta_L \bigl(\bZ^{d-1} \times (-\infty, -n) \bigr) \equiv 0 \pigr) 
    \\ &\qquad\geq \bP_{\lambda}( D_{L,R} \cap F_{R} \neq \emptyset \mid \eta_L \bigl(\bZ^{d-1} \times (-\infty, -n) \bigr) \equiv 0 \pigr)
    \\  &\qquad\geq \sum_{\Delta \subseteq \Lambda_R} \bP_{\lambda}\pigl( D_{L,R} \cap \Delta \neq \emptyset \mid F_R = \Delta, \eta_L \bigl(\bZ^{d-1} \times (-\infty, -n) \bigr) \equiv 0 \pigr)
     \\  &\qquad \qquad\cdot\bP_{\lambda}\pigl( F_R = \Delta \mid \eta_L \bigl(\bZ^{d-1} \times (-\infty, -n) \bigr) \equiv 0 \pigr). 
  \end{split}\end{equation} 
  To lower bound this sum, we will need a few auxiliary inequalities, which we now state and prove. 
  First, by Lemma~\ref{lem:CPdomi}, we note that there are constants $C,c>0$ such that
  \begin{equation}\label{eq:used q}
    \lim_{n \rightarrow \infty} \bP_{\lambda}( |F_R| \geq \rho R/2 \mid \eta_L \bigl(\bZ^{d-1} \times (-\infty, -n) \bigr) \equiv 0 \pigr) \geq 1-Ce^{-cR}.
  \end{equation}
  Next, we claim that for any $\Delta \subseteq \Lambda_R$, 
  \begin{equation}\label{eq:claimEnd}
    \begin{split}
      \lim_{n \rightarrow \infty} \bP_{\lambda}( D_{L,R} \cap \Delta \neq \emptyset \mid F_R = \Delta, \eta_L \bigl(\bZ^{d-1} \times (-\infty, -n) \bigr) \equiv 0 \pigr) \geq \bP_{\lambda}( D_{L,R} \cap \Delta \neq \emptyset ).
    \end{split}
  \end{equation}
  The proof of this statement will be postponed to the end of this proof. 
  We now derive a few more useful inequalities. To this end, let $\Delta \subseteq \Lambda_R$.
  If we let $(\eta_L^{\Delta})$ denote the contact process initiated at time $0$ with $\eta_0(\Delta)\equiv 1$ and $\eta_0(\Delta^c)\equiv 0$, then
  \begin{equation}
    \bP_{\lambda}( D_{L,R} \cap \Delta \neq \emptyset ) = \bP_{\lambda} \bigl( \eta_L^{\Delta} (o)=1 \bigr).
  \end{equation}

  Thus, letting $L \rightarrow \infty$ and defining $\tau^{\Delta} \coloneqq \inf \{ t>0 \colon \eta_t^{\Delta} \equiv 0 \}$, by the complete convergence theorem \cite[Theorem I.2.27]{LiggettSIS1999}, we get
  \begin{equation} \label{eq: something 1}
    \lim_{L \rightarrow \infty} \bP_{\lambda}( D_{L,R} \cap \Delta \neq \emptyset ) = \bP_{\lambda}(\tau^{\Delta}=\infty) \cdot \mu_{\lambda}(X(o)=1).
  \end{equation}
  Moreover, by the self-duality of the contact process and~\cite[Corollary 4.1]{LiggettSteifSD2006}, it holds that
  \begin{equation}\label{eq: something 2}
    \bP_{\lambda}(\tau^{\Delta}=\infty) = 1-\mu_{\lambda}(X\equiv 0 \text{ on } \Delta) \geq 1-(1-\rho)^{|\Delta|}.
  \end{equation} 
  Combining the above inequalities, we obtain
  \begin{equation}\begin{split}
    &\mu_{\lambda}\pigl( X(o) = 1 \mid X\bigl(\bZ^{d-1} \times (-\infty, -n) \bigr) \equiv 0 \pigr) 
    \\  &\qquad\overset{\eqref{eq: long inequality}}{\geq} \sum_{\Delta \subseteq \Lambda_R} \bP_{\lambda}\pigl( D_{L,R} \cap \Delta \neq \emptyset \mid F_R = \Delta, \eta_L \bigl(\bZ^{d-1} \times (-\infty, -n) \bigr) \equiv 0 \pigr)
     \\  &\qquad \qquad\cdot\bP_{\lambda}\pigl( F_R = \Delta \mid \eta_L \bigl(\bZ^{d-1} \times (-\infty, -n) \bigr) \equiv 0 \pigr)
     \\&\qquad\overset{\eqref{eq:claimEnd}}{\geq} \sum_{\Delta \subseteq \Lambda_R} 
     \pigl( \bP_{\lambda}( D_{L,R} \cap \Delta \neq \emptyset ) -o_n(1) \pigr)
     \\  &\qquad \qquad\cdot\bP_{\lambda}\pigl( F_R = \Delta \mid \eta_L \bigl(\bZ^{d-1} \times (-\infty, -n) \bigr) \equiv 0 \pigr)
     \\&\qquad\overset{\eqref{eq: something 1}}{=} \sum_{\Delta \subseteq \Lambda_R} 
     \pigl( \bP_{\lambda}(\tau^{\Delta}=\infty) \cdot \mu_{\lambda}(X(o)=1)-o_L(1) -o_n(1) \pigr)
     \\  &\qquad \qquad\cdot\bP_{\lambda}\pigl( F_R = \Delta \mid \eta_L \bigl(\bZ^{d-1} \times (-\infty, -n) \bigr) \equiv 0 \pigr)
     \\&\qquad\overset{\eqref{eq: something 2}}{=} \sum_{\Delta \subseteq \Lambda_R} 
     \pigl( \mu_{\lambda}(X(o)=1)\bigl( 1-(1-\rho)^{|\Delta|}\bigr)-o_L(1) -o_n(1) \pigr)
     \\  &\qquad \qquad\cdot\bP_{\lambda}\pigl( F_R = \Delta \mid \eta_L \bigl(\bZ^{d-1} \times (-\infty, -n) \bigr) \equiv 0 \pigr)
     \\&\qquad\geq
     \pigl( \mu_{\lambda}(X(o)=1)\bigl( 1-(1-\rho)^{\rho R/2}\bigr)-o_L(1) -o_n(1) \pigr)
     \\  &\qquad \qquad\cdot\bP_{\lambda}\pigl( |F_R| \geq \rho R/2 \mid \eta_L \bigl(\bZ^{d-1} \times (-\infty, -n) \bigr) \equiv 0 \pigr)
     \\&\qquad\overset{\eqref{eq:used q}}{\geq}
     \pigl( \mu_{\lambda}(X(o)=1)\bigl( 1-(1-\rho)^{\rho R/2}\bigr)-o_L(1) -o_n(1) \pigr)
     (1-Ce^{-cR}).
  \end{split}\end{equation} 
  Since the above inequalities holds for all \( L \in \mathbb{N}\), letting \( n \to \infty, \) we conclude that  
  \begin{align}
    \lim_{n \rightarrow \infty} &\mu_{\lambda}\pigl( X(o) = 1 \mid X\bigl(\bZ^{d-1} \times (-\infty, -n) \bigr) \equiv 0 \pigr) 
    \\  &\qquad\geq \mu_{\lambda}\bigl(X(o)=1 \bigr) \bigl(1-(1-\rho)^{\rho R/2} \bigr)(1-Ce^{-cR})
  \end{align}
  and, by letting $R\rightarrow \infty$, that \eqref{eq: goal in last proof} holds.

  What remains is to show that \eqref{eq:claimEnd} holds. To this end, recall that $\Delta \subseteq \Lambda_R$. For \( n \in \mathbb{N} \) and \( L \in [0,\infty)\), consider the (random) set $E_{n,L}$ of vertices $y\in \bZ^{d-1} \times (-\infty,-n) \subset \bZ^d$ to which there exists an active path within the time-window $[0,L]$ that ends at $(y,L)$ and that originates from some space-time location within $\Lambda_{n/2} \times [0,L]$.
  Since $\{E_{n,L} \neq \emptyset\}$ is an increasing event, using first the d-FKG property and then the inequality~\eqref{eq:LDB_GM14} %
  together with a union bound, we have that 
  \begin{align} 
    &\bP_{\lambda} \pigl( E_{n,L} \neq \emptyset \mid F_R = \emptyset, \eta_L\bigl(\bZ^{d-1} \times (-\infty,-n) \bigr) \equiv 0 \pigr)
    \leq \bP_{\lambda} ( E_{n,L} \neq \emptyset )
    \leq Ce^{-cn}
  \end{align}
  for some constants $C,c\in (0,\infty)$ depending on $\lambda$, $d$, and $L$. Hence 
  \begin{equation}\label{eq:boundONe_nl}
    \lim_{n \rightarrow \infty} 
    \bP_{\lambda} \pigl( E_{n,L} \neq \emptyset \mid F_R = \emptyset, \eta_L\bigl(\bZ^{d-1} \times (-\infty,-n) \bigr) \equiv 0 \pigr) =0.
  \end{equation}
  Further, since the event $\{D_{L,R} \cap \Delta \neq \emptyset \}$ is increasing, we have  
  \begin{equation}\label{eq: eq in end}\begin{split}
    & \bP_{\lambda} \pigl( D_{L,R} \cap \Delta \neq \emptyset \mid F_R = \Delta, \eta_L \bigl(\bZ^{d-1} \times (-\infty, -n) \bigr) \equiv 0 \pigr)
    \\&\qquad\geq \bP_{\lambda} \pigl( D_{L,R} \cap \Delta \neq \emptyset \mid F_R = \emptyset, \eta_L \bigl(\bZ^{d-1} \times (-\infty, -n) \bigr) \equiv 0 \pigr)
    \\&\qquad\geq \bP_{\lambda} \pigl( D_{L,R} \cap \Delta \neq \emptyset \mid E_{n,L} = \emptyset, F_R = \emptyset, \eta_L \bigl(\bZ^{d-1} \times (-\infty, -n) \bigr) \equiv 0 \pigr)
    \\&\qquad\qquad\cdot \bP_{\lambda}\pigl( E_{n,L} = \emptyset \mid F_R = \emptyset, \eta_L \bigl(\bZ^{d-1} \times (-\infty, -n) \bigr) \equiv 0 \pigr),
  \end{split}\end{equation}
  where we used the d-FKG property in the first inequality.  
  Now note that on the event $\{ E_{n,L} = \emptyset\}$, the events \( \{F_R = \emptyset \} \) and \( \{ \eta_L \bigl(\bZ^{d-1} \times (-\infty, -n) \bigr) \equiv 0 \} \) do not influence the graphical construction within the space-time region $\Lambda_{n/2} \times [0,L].$ Since $D_{L,R}$ depends only on the graphical construction within this region, it follows that
  \begin{equation}\label{eq: last eq}\begin{split}
    &\lim_{n \rightarrow \infty} \bP_{\lambda}( D_{L,R} \cap \Delta \neq \emptyset \mid E_{n,L} = \emptyset, F_R = \emptyset, \eta_L \bigl(\bZ^{d-1} \times (-\infty, -n) \bigr) \equiv 0 \pigr) 
    \\&\qquad= \bP_{\lambda}( D_{L,R} \cap \Delta \neq \emptyset ).
  \end{split}\end{equation} 
  Combining~\eqref{eq:boundONe_nl},~\eqref{eq: eq in end}, and~\eqref{eq: last eq}, we obtain~\eqref{eq:claimEnd} as desired.
 \end{proof}

\bibliographystyle{plainnat}

\end{document}